\documentclass[12pt,reqno,english]{amsart}

\usepackage{amsmath, amssymb, amscd, amsthm, amsfonts}
\usepackage{graphicx,xcolor,color}
\usepackage[linktocpage,colorlinks,linkcolor=blue,anchorcolor=blue,citecolor=blue,urlcolor=blue,pagebackref]{hyperref}
\usepackage{mathtools,enumitem}
\usepackage[T1]{fontenc}
\usepackage[utf8]{inputenc}
\usepackage{float}
\newcommand{\kb}[1]{{\color{red}\bf[KB: #1]}}

\numberwithin{equation}{section}
\theoremstyle{plain}
\usepackage[letterpaper]{geometry}
\renewcommand*{\backref}[1]{\ifx#1\relax \else Page #1 \fi}
\renewcommand*{\backrefalt}[4]{%
  \ifcase #1 \footnotesize{(Not cited.)}%
  \or        \footnotesize{(Cited on page~#2.)}%
  \else      \footnotesize{(Cited on pages~#2.)}%
  \fi
}

\title[Normal Approximation for Geometric and Topological Statistics ]{A Flexible Approach for Normal Approximation of Geometric and Topological Statistics}

\author{Zhaoyang Shi}
\author{Krishnakumar Balasubramanian}
\author{Wolfgang Polonik}

\address{\hspace{-0.16in}Department of Statistics, University of California, Davis.\newline Email: \texttt{zysshi@ucdavis.edu}, \texttt{ kbala@ucdavis.edu}, \texttt{wpolonik@ucdavis.edu}}
\usepackage[foot]{amsaddr}

\newtheorem{Theorem}{Theorem}[section]
\newtheorem{Assumption}{Assumption}[section]
\newtheorem{Proposition}{Proposition}[section]
\newtheorem{Remark}{Remark}[section]
\newtheorem{Example}{Example}[section]
\newtheorem{Corollary}{Corollary}[section]
\newtheorem{Definition}{Definition}[section]
\newtheorem{Lemma}{Lemma}[section]

\newcommand{\mbb}{\mathbb}

\newcommand{\mcal}{\mathcal}

\DeclareMathOperator{\Var}{Var}

\allowdisplaybreaks

\begin{document}

\maketitle

\begin{abstract}
We derive normal approximation results for a class of stabilizing  functionals of binomial or Poisson point process, that are not necessarily expressible as sums of certain score functions. Our approach is based on a flexible notion of the add-one cost operator, which helps one to deal with the second-order cost operator via suitably appropriate first-order operators. We combine this flexible notion with the theory of strong stabilization to establish our results. We illustrate the applicability of our results by establishing normal approximation results for certain geometric and topological statistics arising frequently in practice. Several existing results also emerge as special cases of our approach.
\end{abstract}


\section{Introduction}

Let $(\mbb{X},\mcal{F})$ be a metric measure space equipped with a $\sigma$-finite measure $\mbb{Q}$ and a metric $d: \mbb{X}\times \mbb{X}\rightarrow [0,\infty)$. For $s \ge 1$, let $\mcal{P}_{s}$ denote the canonical Poisson process on $\mbb{X}$ with intensity measure $\lambda:=s\mbb{Q}$, and for $\mbb{Q}$ a probability measure, let $\xi_{n}$ denote the binomial process associated to $\mbb{Q}$. Let  $d_{K}(Y,Z)$ denote the Kolmogorov distance between two random variables $Y,Z$, i.e., $d_{K}(Y,Z):=\sup_{t\in\mbb{R}}|~\mbb{P}(Y\le t)-\mbb{P}(Z\le t)|$. In this work, we study normal approximation results for real-valued functionals $F_{s}(\mcal{P}_{s})$ and $F_{n}(\xi_{n})$ respectively of the Poisson and the binomial point processes in the Kolmogorov metric under relatively flexible assumptions on the functionals. In particular, motivated by geometric and topological  statistics, we focus on the case when the functionals $F_s$ and $F_n$ are not necessarily expressible as sums of certain score functions, and on obtaining presumably optimal bounds in this case.

Our proof techniques are based on the widely-used concept of \emph{stabilization}. Indeed, since the use of stabilization concepts to establish central limit theorems for Poisson-based minimal spanning tree in~\cite{kesten1996central}, these concepts have been widely developed as a general tool to establish normal approximation rates for various functionals of Poisson and Binomial point process. We refer the interested reader to~\cite{penrose2001central, penrose2005multivariate, baryshnikov2005gaussian, penrose2005normal,schreiber2010limit, last2016normal, chatterjee2017minimal,lachieze2017new} and reference therein for details. In particular,~\cite{last2016normal} develops normal approximation bounds for a fairly general class of functions of Poisson processes by combining Malliavin-Stein techniques~\cite{peccati2010stein,peccati2016stochastic}, second-order Poincar\'e inequalities~\cite{chatterjee2009fluctuations,nourdin2009second} and stabilization concepts, and by using the iterated add-one cost operator, also called second order cost operator. Considering the case of functionals expressible as a sum of exponentially stabilizing score functions,~\cite{lachieze2019normal} establishes user-friendly normal approximation results based on~\cite{last2016normal}.  The work of~\cite{lachieze2020quantitative} introduces bounds for general functionals of Poisson process. Their method does not involve the hard-to-evaluate iterated add-one cost operators but uses the add-one cost operator at two different scales, an approach pioneered by~\cite{chatterjee2017minimal} for the case of Poisson-based minimal spanning trees. However, their generality comes at the cost of sub-optimality -- in general, the bounds based on~\cite{lachieze2020quantitative} are sub-optimal compared to those of~\cite{last2016normal}. Furthermore, for the case of functionals that are expressible as a sum of exponentially stabilizing score functions, the bounds of~\cite{lachieze2020quantitative} necessarily lead to sub-optimal rates. 

Hence, the following question remains: \emph{Can one obtain presumably optimal bounds for general functionals that automatically result in presumably optimal bounds when specialized to the case of  functionals that can be expressed as sums of score functions}. Following~\cite{lachieze2019normal}, we use the term \emph{presumably optimal} to refer to the case when the order of the normal approximation is the same as that of a sum of i.i.d. random variables. In this work, we answer this question in the affirmative for a class of functionals. Similar to~\cite{lachieze2020quantitative}, our approach is based on the idea of using the add-one cost operator at two scales. However, in contrast to their work, we use it to directly simplify the evaluation of the iterated add-one cost operators. When specialized to the case of sums of score functions, such an approach recovers the presumably optimal results of~\cite{lachieze2019normal}. To summarize, we make the following contributions:
\begin{itemize}[noitemsep]
    \item In Definition~\ref{defflex}, we introduce a flexible notion of the add-one cost operator with a general set $A_{x}$ that allows to (relatively) easily evaluate computations with the iterated add-one cost or second-order difference operators, for general functionals of Poisson and binomial point process that are not necessarily a sum. 
    \item In our main results, Theorems~\ref{MTPoi} and~\ref{MTbino}, we provide normal approximation results for functionals of Poisson and binomial point processes respectively. In particular, the functionals do not necessarily need to be expressible as sums of certain score functions. 
    \item We illustrate the applicability of our approach by deriving normal approximation results for several geometric and topological statistics. Specifically, in Theorem~\ref{totalknn} and \ref{weightedknn} we use our approach to derive normal approximation results for the total edge length of $k$-Nearest Neighbor graph and weighted $k$-Nearest Neighbor graph based Shannon entropy estimators. In Theorem~\ref{thm:eulerc}, we derive results for Euler Characteristic, which is an elementary statistics widely used in the field of topological data analysis. Finally, we discuss the applicability of our approach for the minimal spanning tree problem in Theorem~\ref{mst}, by recovering existing results via our approach.
\end{itemize}

\textbf{Organization.} The rest of the paper is organized as follows. In Section~\ref{sec:prelim}, we introduce the basics of point processes, stabilization concepts and required assumptions. In Section~\ref{sec:main}, we present our main theorems and discuss relations to existing results. In Section~\ref{sec:app}, we discuss applications of our results to geometric and topological statistics. The proofs are provided in Section~\ref{proofs} and \ref{proof2}.

\section{Preliminaries}\label{sec:prelim}

\subsection{Point Process Basics}

Let $(\mbb{X},\mcal{F})$ be a measure space with a $\sigma$-finite measure $\mbb{Q}$ and a metric $d: \mbb{X}\times \mbb{X}\rightarrow [0,\infty)$. Let \textbf{N} be the set of $\sigma$-finite counting measures on $\mbb{X}$, which can be interpreted as point configurations in $\mbb{X}$. Thus, we treat the elements from \textbf{N} as sets. The set \textbf{N} is equipped with the smallest $\sigma$-field $\mcal{N}$ such that the maps $m_{A}:\textbf{N}\rightarrow \mbb{N}\cup\{0,\infty\},\mcal{M}\mapsto\mcal{M}(A)$ are measurable for all $A\in\mcal{F}$; see \cite{kallenberg1997foundations} and \cite{last2011poisson}. A point process $\eta$ is a random element in \textbf{N}. Denote by $\textbf{F}(\textbf{N})$ the class of all measurable functions $f:\textbf{N}\rightarrow \mbb{R}$, and by $L^{0}(\mbb{X}):=L^{0}(\mbb{X},\mcal{F})$ the class of all real-valued, measurable functions $F$ on $\mbb{X}$. Note that, as $\mcal{F}$ is the completion of $\sigma(\eta)$, each $F\in L^{0}(\mbb{X})$ can be written as $F=f(\eta)$ for some measurable function $f\in \textbf{F}(\textbf{N})$. Such a mapping $f$, called a \textit{representative} of $F$, is $\mbb{Q}\circ \eta^{-1}$-a.s. uniquely defined. In order to simplify the discussion, we make this convention: whenever a general function $F$ is introduced, we will select one of its representatives and denote such a representative mapping by the same symbol $F$. Throughout this paper, we denote by $L_{\eta}^{2}(\mbb{X})$ the space of all square-integrable functions $F$ of a point process $\eta$ with $\mbb{E}F^{2}(\eta)<\infty$. We mainly consider two different classes of point processes: Poisson point process and binomial point process. 


\begin{Definition}[Poisson Point Process]
A Poisson point process with intensity measure $\lambda$ is a point process $\mcal{P}(\lambda)$ on $\mbb{X}$ with the following two properties:
\begin{enumerate}[noitemsep]
    \item  $\forall B\in \mathcal{F}$, $\mcal{P}(\lambda)(B)$ is a Poisson random variable with parameter $\lambda(B)$.
 	\item $\forall m\in \mathbb{N}_{+}$ and for any pairwise disjoint sets $B_{1},B_{2},...,B_{m}\in \mcal{F}$, we have that the random variables $\mcal{P}(\lambda)(B_{1}),\mcal{P}(\lambda)(B_{2}),\ldots ,\mcal{P}(\lambda)(B_{m})$ are independent.
\end{enumerate}	
\end{Definition}

\begin{Definition}[Binomial Point Process]
	Let $P$ be a probability distribution and $n$ be a fixed positive integer. Let $X_{1},X_{2},...,X_{n}$ be i.i.d. random variables sampled from $P$. The binomial point process $\xi_{n}$ based on $P$ and $n$ is defined as $\xi_n:=\sum_{i=1}^{n}\delta_{X_{i}},$ where $\delta$ is the Dirac measure.
\end{Definition}

We now describe the setting for developing normal approximations of functionals of Poisson and binomial point processes. For $s\ge 1$, let $\lambda:=s\mbb{Q}$ be the intensity measure of Poisson point process $\mcal{P}(\lambda) \coloneqq \mcal{P}_{s}$. For the case when $\mbb{Q}$ is the probability measure, let $\xi_{n}$ be the binomial point process based on $\mbb{Q}$ and $n$. Consider square-integrable functionals of these two point processes, i.e., $F(\mcal{P}_{s}):=F_{s}(\mcal{P}_{s})\in L_{\mcal{P}_{s}}^{2}(\mbb{X})$ and $F(\xi_{n}):=F_{n}(\xi_{n})\in L_{\xi_{n}}^{2}(\mbb{X})$. We then seek upper bounds for the following two quantities:
\begin{equation*}
	d_{K}\left(\frac{F_{s}(\mcal{P}_{s})-\mbb{E}F_{s}(\mcal{P}_{s})}{\sqrt{\Var F_{s}(\mcal{P}_{s})}},N\right), \qquad\text{and}\qquad	d_{K}\left(\frac{F_{n}(\xi_{n})-\mbb{E}F_{n}(\xi_{n})}{\sqrt{\Var F_{n}(\xi_{n})}},N\right),
\end{equation*}
where $N$ is the standard normal random variable.

\subsection{Stabilization} The notion of stabilization is widely used in deriving normal approximation rates for functionals of Poisson or binomial point processes~\cite{kesten1996central,penrose2001central, penrose2005multivariate, penrose2005normal}. We start with introducing notions of  stabilization for functionals which are not necessarily representable as sums of sore functions. 

\begin{Definition}[Add-one Cost Operator]\label{addonecost}
	Let $F$ be a measurable functional of a point process $\eta$ on $(\mbb{X},\mcal{F})$. The family of Add-One Cost Operators, $D=(D_{x})_{x\in \mbb{X}}$, are defined as 
	\begin{align*}
	D_{x}F(\eta):=F(\eta\cup\{x\})-F(\eta).
	\end{align*}
	Similarly, we can define a second-order cost operator (also called iterated add-one cost operator): for any  $x_{1},x_{2}\in\mbb{X}$,
	\begin{align*}
	D_{x_{1},x_{2}}F(\eta):=F(\eta\cup\{x_{1},x_{2}\})-F(\eta\cup\{x_{1}\})-F(\eta\cup\{x_{2}\})+F(\eta).
	\end{align*}
\end{Definition}
In addition, we define for any $y\in\mbb{X}$,
\begin{align*}
    D_{x}F^{y}(\eta):=F(\eta\cup\{y\}\cup\{x\})-F(\eta\cup\{y\}).
\end{align*}
Clearly, when $\{y\}=\emptyset$, it degenerates into the add-one cost operator.

Based on the add-one cost operator introduced above, we next introduce weak and strong stabilization in the context of functionals $F_{s}$ of the Poisson point process $\mcal{P}_s$. Similar definitions hold automatically for the case of functionals $F_{n}$ of the binomial Point Process $\xi_n$ .

\begin{Definition}[Weak Stabilization]
	The functional $F_{s}$ is said to be weakly stabilizing at $x\in \mbb{X}$, if and only if there exists a random variable $\Delta_{x}$ such that for any sequence $(W_{m})_{m\ge 1}$ in $\mcal{F}$ tending to $\mbb{X}$, as $m\rightarrow \infty$,  we have $D_{x}F_{s}(W_{m})\rightarrow \Delta_{x}$, almost surely.
\end{Definition}



\begin{Definition}[Strong Stabilization]\label{strongs}
	The functional $F_{s}$ is said to be strongly stabilizing at $x\in \mbb{X}$, if and only if there exists an almost surely finite random variable $R_{x}$, which is referred to as the radius of stabilization, such that for all finite $\mbb{A}\subset \mbb{X} \backslash B_{x}(R_{x})$, with probability 1,
	\begin{align*}
	D_{x}F_{s}((\mcal{P}_{s}\cap B_{x}(R_{x}))\cup \mbb{A})=D_{x}F_{s}(\mcal{P}_{s}\cap B_{x}(R_{x})),
	\end{align*} 
	where $B_{x}(R_{x})\coloneqq \{y\in\mbb{X}:d(x,y)\le R_{x}\}$.
\end{Definition}

Clearly, strong stabilization implies weak stabilization. In some cases, the functionals $F_{s}$ and $F_{n}$ can be represented as a sum of the form
\begin{equation}\label{scoreform}
F_{s}(\mcal{P}_{s}):=\sum_{x\in \mcal{P}_{s}}f_{s}(x,\mcal{P}_{s}),\qquad \text{and}\qquad
F_{n}(\xi_{n}):=\sum_{x\in\xi_{n}}f_{n}(x,\xi_{n}),
\end{equation}
where $f_{s},f_{n}$ are called {\em score functions}. In this case, there exists a useful notion of stabilization based on the score functions. For simplicity, we still state the definition for functionals of Poisson point process; the binomial case is defined similarly.


\begin{Definition}[Score-based Stabilization~\cite{lachieze2019normal}]
		The score function $f_{s}$ is said to be stabilizing at $x\in\mbb{X}$, if and only if there exists an almost surely finite random variable $R_{x}$ (the radius of stabilization) such that for all finite $\mbb{A}\subset \mbb{X} \backslash B_{x}(R_{x})$, we have
\begin{align*}
f_{s}(x,(\mcal{P}_{s}\cap B_{x}(r_{x}))\cup \mbb{A})=f_{s}(x,\mcal{P}_{s}\cap B_{x}(r_{x})).
\end{align*}
\end{Definition}
Informally speaking, the above definition posits that the value of the score function $f_{s}$ will not be affected by the points outside the ball centered at $x$ with radius $R_{x}$. For discussing the relation between strong stabilization and score-based stabilization, we present the following simple result.

\begin{Proposition}\label{prop 2.1}
	Given any $F_{s}=\sum_{y\in \mcal{P}_{s}}f_{s}(y,\mcal{P}_{s})$, we have for all $x\in\mbb{X}$, 
	\begin{align}\label{Ff}
	D_{x}F_{s}(\mcal{P}_{s})=f_{s}(x,\mcal{P}_{s}\cup \{x\})+\sum_{y\in\mcal{P}_{s}}D_{x}f_{s}(y,\mcal{P}_{s}). 
	\end{align}	
\end{Proposition}
\begin{proof}[Proof of Proposition~\ref{prop 2.1}]
We have 
	\begin{align}\label{3f}
	D_{x}F_{s}(\mcal{P}_{s})&=F_{s}(\mcal{P}_{s}\cup\{x\})-F_{s}(\mcal{P}_{s}) \nonumber \\
	&=\sum_{y\in\mcal{P}_{s}\cup\{x\}}f_{s}(y,\mcal{P}_{s}\cup\{x\})-\sum_{y\in \mcal{P}{s}}f_{s}(y,\mcal{P}_{s})\nonumber \\
	&=f_{s}(x,\mcal{P}_{s}\cup\{x\})+\left(\sum_{y\in\mcal{P}_{s}}f_{s}(y,\mcal{P}_{s}\cup\{x\})-\sum_{y\in\mcal{P}_{s}}f_{s}(y,\mcal{P}_{s})\right)\\
	&=f_{s}(x,\mcal{P}_{s}\cup\{x\})+\sum_{y\in\mcal{P}_{s}}D_{x}f_{s}(y,\mcal{P}_{s})\nonumber, 
	\end{align}
which completes the proof. 
\end{proof}
\begin{Remark} We now make the following remarks.
\begin{itemize}[noitemsep]
\item[(1)] Strong stabilization focuses on the cost function of the functional $F$ while the score-based stabilization is assumed on the score functions $f$. Equation \eqref{3f} reveals this relationship. It plays an important role in Section \ref{Shannon app} and its proof.
\item[(2)] Strong stabilization is more general than score-based stabilization in that strong stabilization does not restrict the form of the functional to be expressible as a sum of scores. Furthermore, the same functional might be expressible in multiple ways as sums of scores. Depending on the representation, it might be easier or harder to compute the radius of stabilization and also the moments of the score functions (which also play a crucial role - see below). Strong stabilization, however, provides an approach to directly work with the functional $F$ itself. 
\end{itemize}
\end{Remark}


The following example from the literature on Topological Data Analysis (TDA), further illustrates the aforementioned remarks. Readers unfamiliar with the basics of TDA are directed to the elementary definitions provided in Appendix~\ref{sec:tda}. We also refer to~\cite{edelsbrunner2010computational, boissonnat2018geometric} for more on the basics of TDA.

 
\begin{Example}[Euler Characteristic]	\label{egec}
	Given a simplicial complex $K$, the Euler characteristic is defined as
	\begin{equation*}
	\chi(K):=\sum_{k=0}^{\infty}(-1)^{k}\#\{S_{k}\},
	\end{equation*}
	where $\#\{S_{k}\}$ is the number of simplices of dimension $k$. 
\end{Example}
Typically, the simplicial complex $K$, is taken to be the Vietoris-Rips complex (VR complex) or the \v{C}ech complex, constructed over a point cloud sampled from binomial or Poisson point processes $\xi_n$ or $\mathcal{P}_s,$ respectively. In this case, we denote the simplicial complex as $K(\xi_n)$ or $K(\mathcal{P}_s)$ to denote the dependency on the underlying point process explicitly. We now discuss the stabilization properties of the above statistic. While it is possible to express the Euler characteristic as a sum of certain score functions, it is not required to do so, as the Euler characteristic is strongly stabilizing with radius of stabilization $R_{x}=2r$ for the \u Cech and the VR-complex; see \cite{krebs2021approximation}. 

The following example, on the total edge length of a $k$-nearest neighbor graphs, is a canonical example of a geometric statistic that satisfies score-based stabilization and strong stabilization.

\begin{Example}[$k$-Nearest Neighbor ($k$-NN) Graphs]\label{eg:knn}
	Consider a configuration of a Poisson point process $\mcal{P}_{s}$, where here we represent $\mcal{P}_s$ by a random number of (conditionally) i.i.d. points $X_i$, i.e. $\mcal{P}_{s}=\{X_{i}\}_{i=1}^{|\mcal{P}_{s}|}$. For some $k\in \mbb{N}_{+}$, and for every integer $1\le j\le k$, denote by $X_{j,i}$ the $j$-nearest neighbor of $X_{i},$ i.e. $X_{j,i}$ is the $j$th closest point to $X_{i}.$ Furthermore, let $\rho_{j,i}$ denote the distance between $X_{j,i}$ and $X_{i}$. Then, the (undirected) $k$-NN graph $NG_{k}(\mcal{P}_{s})$ is the graph with the vertex set $V:=\mcal{P}_{s}$ and an edge $x\sim y$ if $y$ is some $j$-nearest neighbor of $X$ and (or) $x$ is some $j$-nearest neighbor of $y$. For $\vartheta>0$, we define 
	\begin{align}\label{eq:fsforknn}
	f_{s}(x,\mcal{P}_{s}):=
    \left\{
    \begin{aligned}
    &\sum_{x\sim y}\frac{1}{2}d(x,y)^{\vartheta}, \quad \text{if $x,y$ are mutual  $k$-nearest neighbors},\\
    &\sum_{x\sim y}d(x,y)^{\vartheta}, \quad \text{if $x,y$ are not mutual  $k$-nearest neighbors}.
    \end{aligned}
    \right.
	\end{align}
The total edge length is defined as
	\begin{equation*}
	F_{s}:=\sum_{x\in \mcal{P}_{s}}f_{s}(x,\mcal{P}_{s}).
	\end{equation*}
\noindent According to \cite{lachieze2019normal}, the total edge length statistic satisfies score-based stabilization. Additionally, by the proof of \cite[Lemma 6.1]{penrose2001central}, we also have that it satisfies strong stabilization with the radius of stabilization being $R_{x}=4R$, where $R$ is defined in the following way: for each $t>0$, construct six disjoint equilateral triangles $T_{j}(t)$, $1\le j\le 6$, such that the origin is a vertex of each triangle, such that each triangle has edge length $t$ and such that $T_{j}(t)\subset T_{j}(u)$ whenever $t<u$. Then, define $R$ to be the minimum $t$ such that each triangle $T_{j}(t)$ contains at least $k
+1$ points from $\mcal{P}_{s}$.
\end{Example}
	

By definition, strong stabilization only focuses on the first-order add-one cost operator. In order to deal with second-order cost operators, which are also crucial in obtaining our normal approximation results, we introduce the following flexible notion of add-one cost operators.

\begin{Definition}[Flexible Add-One Cost]\label{defflex}
	For any point process $\eta$ in $(\mbb{X},\mcal{F})$, any $x\in\mbb{X}$ and a set $A_{x}\in\mcal{F}$ (that may or may not depend on $x$), the flexible add-one cost operator for the functional $F$ is defined as
$$D_{x}F(A_{x}):=D_{x}F(A_{x})(\eta):=D_{x}F(\eta|_{A_{x}}):=F(\left(\eta|A_{x}\right)\cup\{x\})-F(\eta|A_{x}),$$
where we denote by $\eta|A_{x}$ the restriction of the point process $\eta$ to the set $A_{x}$ (see, for example,~\cite{lachieze2020quantitative}). 
\end{Definition}
Informally speaking, we introduce the flexible add-one cost $D_{x}F(A_{x})$ by only observing the point process in the `window' $A_x$. Obviously, if one sets $A_{x}=\mbb{X}$, the flexible add-one cost function degenerates into the classical add-one cost function in Definition \ref{addonecost}. The following proposition, whose proof is immediate by simply using the definition of the second-order cost function, provides a way to deal with the second-order cost function by the flexible cost function defined above.
\begin{Proposition}\label{2ndcost}
	Under the setting of Definition \ref{addonecost} and \ref{defflex}, we have
    \begin{equation*}
	D_{x_{1},x_{2}}F=\left(D_{x_{1}}F^{x_{2}}-D_{x_{1}}F^{x_{2}}(A_{x})\right)+\left(D_{x_{1}}F^{x_{2}}(A_{x})-D_{x_{1}}F(A_{x})\right)+\left(D_{x_{1}}F(A_{x})-D_{x_{1}}F\right).
	\end{equation*}
    Particularly, when $A_{x}=\mbb{X}$, we have 
	\begin{equation*}
	D_{x_{1},x_{2}}F=D_{x_{1}}F^{x_{2}}-D_{x_{1}}F.
	\end{equation*}
\end{Proposition}

\subsection{Assumptions}
We now discuss the assumptions made in our work to obtain the normal approximation results. On the measure $\mathbb{Q}$, following~\cite{penrose2007gaussian,penrose2005normal,yukich2015surface,lachieze2019normal}, we make the following assumption: There exist constants $\kappa>0,\omega>1$ such that for $r\ge 0$ and all $x\in\mbb{X}$,
	\begin{align}\label{assQ}
	\underset{\epsilon\rightarrow 0^{+}}{\limsup}~\frac{\mbb{Q}(B_{x}(r+\epsilon)) - \mathbb{Q}(
	B_{x}(r))}{\epsilon}\le \kappa\omega r^{\omega-1}, \quad \mbb{Q}(\{x\})=0.
	\end{align}
For example, one can consider a measure $\mbb{Q}$ on $\mbb{X}$, a full dimensional subset of $\mbb{R}^{d}$, with a bounded density with respect to the Lebesgue measure, where one can choose $\kappa = \sup f$ and $\omega = d$.
We also make the following tail-bound assumption on the radius of strong stabilization.

\begin{Assumption}[Decay of Radius of Stabilization]\label{decay}
	Under the setting of  strong stabilization and \eqref{assQ}, we say the radius of stabilization $R_{x}$ decays exponentially if and only if 
	there exist constants $c_{1},c_{2},c_{3}>0$ such that for $r\ge 0$,
	\begin{equation*}
	\mathbb{P}(R_{x}\ge r)\le c_{1} e^{-c_{2}(s^{1/\omega} r)^{c_{3}}}.
	\end{equation*}
If $R_x$ is based on a binomial process $\xi_n$, then a similar decay holds with $s$ replaced by $n$.
\end{Assumption}

Yet another reason for why we refer to Definition~\ref{defflex} as ``flexible'' is that even when the tail probability of the radius of stabilization $R_{x}$ is unknown for a specific functional, it might be possible to pick $A_{x}$ ``strategically'' and use our approach to obtain normal approximation bounds. We illustrate this point in Section~\ref{sec:mst} by using our approach to recover existing results on normal approximation for the total edge length of the minimal spanning tree.

We next move on to the assumptions on the (flexible) add-one cost operator. Throughout the paper, we assume that $\mathbb{E}\int(D_{x}F)^{2}\lambda(dx)<\infty$. Furthermore, we assume the following $\mathbb{K}$-exponential bound assumption.

\begin{Assumption}[$\mathbb{K}$-exponential bound]\label{kexp}
	We say the add-one cost function $D_{x}F$ satisfies a $\mathbb{K}$-exponential bound, where $\mbb{K}$ is a measurable subset of $\mbb{X}$, if and only if for $x,x^{*}\in\mbb{X}$, there exist constants $k_{1},k_{2},k_{3}>0$ such that
    \begin{equation*}
	\mathbb{P}(D_{x}F\neq 0)\le k_{1}e^{-k_{2}d_{s}(x,\mbb{K})^{k_{3}}},\qquad \text{and}\qquad \mathbb{P}(D_{x}F^{x^{*}}\neq 0)\le k_{1}e^{-k_{2}d_{s}(x,\mbb{K})^{k_{3}}},
	\end{equation*} 
    where $d_{s}(\cdot,\cdot):=s^{1/\omega}d(\cdot,\cdot)$ and $d(x,\mathbb{K})\coloneqq \inf_{y\in \mathbb{K}}~d(x,y)$. 
 Similarly, we can assume the above for binomial point processes by changing $s$ as $n$.
\end{Assumption}

A similar assumption has been made in~\cite[Equations (2.8) and (2.9)]{lachieze2019normal} on the score functions to capture functionals whose variances exhibit surface area order scaling. Here, we make the assumption directly on the functional $F$, which captures a more general class of functionals than that considered in~\cite{lachieze2019normal}.

\begin{Assumption}[Moment Condition]\label{fmc}
We say the functional $F_{s}$ satisfies the moment condition if and only if there exists some $p>0$ and $H<\infty$ such that
\begin{equation}\label{H2}
\underset{s\ge 1}{\sup}~\underset{x,x^{*}\in\mbb{X}}{\sup}(\mathbb{E}~|D_{x}F_{s}|^{p}+\mathbb{E}|D_{x}F_{s}^{x^{*}}|^{p})=H.
\end{equation}
If the Poisson process ${\mathcal P}_s$ is replaced by a binomial process $\xi_n,$ then the above suprema are taken over $n$ rather than $s$ as well as the functional $F_{s}$ is changed to $F_{n}$.
\end{Assumption}
Bounded moment conditions are commonly made to derive normal approximation results. For related work in the context of stabilizing functionals of point process, see \cite[Equations (2.6) and (2.7)] {lachieze2019normal} and \cite[Equations (1.5) and (1.8)]{lachieze2020quantitative}. While \cite{lachieze2019normal} considers moment conditions on score functions, we directly deal with the functional $F$ so that it fits a more general class. 

\section{Main Results}\label{sec:main}
We now present our two main results on the normal approximation of a certain class of functionals of Poisson and binomial point process, Theorem \ref{general} and \ref{MTbino} respectively, that are not necessarily expressible as sums of score functions. We discuss several applications in Section~\ref{sec:app}. Firstly, we introduce a general result for functionals of Poisson point process. We remark that the following theorem does not leverage Assumptions \ref{decay}, \ref{kexp} and \ref{fmc}. However, in Corollary~\ref{MTPoi} we present a refined results under the above mentioned set of assumptions.

\begin{Theorem}[Normal Approximation for Functionals of Poisson Point Processes]\label{general}
Let $F$ be a functional of the Poisson point process $\mcal{P}(\lambda)$ with $F\in L_{\mcal{P}(\lambda)}^{2}$ and $\mathbb{E}\int(D_{x}F)^{2}\lambda(dx)<\infty$. For any $x,x_{1},x_{2}\in\mbb{X}$, define
\begin{align}\label{firstorder}
\mbb{E}|D_{x}F-D_{x}F(A_{x})|^{4} := b_{1}(x,A_{x}),\quad \mbb{E}|D_{x}F(A_{x})|^{4}:= b_{2}(x,A_{x}),
\end{align}
and
\begin{align}\label{secondorder1}
\mbb{E}|D_{x_{1}}F^{x_{2}}-D_{x_{1}}F^{x_{2}}(A_{x_{1}})|^{4}&:= b_{3}(x_{1},x_{2},A_{x_{1}}),\\
\mbb{E}|D_{x_{1}}F(A_{x_{1}})-D_{x_{1}}F|^{4}&:= b_{4}(x_{1},x_{2},A_{x_{1}})\label{secondorder2},\\ \mbb{E}|D_{x_{1}}F^{x_{2}}(A_{x_{1}})-D_{x_{1}}F(A_{x_{1}})|^{4}&:= b_{5}(x_{1},x_{2},A_{x_{1}})\label{secondorder3}.
\end{align}
Then, there is an absolute constant $C^*>0$ such that 
\begin{align*}
	d_{K}\Big(\frac{F-\mathbb{E}F}{\sqrt{\Var F}},N\Big)\le C^{*}\sum_{i=1}^6\gamma_{i}',
\end{align*}
where 
\begin{align*}
    \gamma_{1}&':=\frac{1}{\Var F} \Big(\int \Big(\sum_{j=1}^{2}b_{j}(x_{1},A_{x_{1}})^{\frac{1}{4}}\sum_{j=1}^{2}b_{j}(x_{2},A_{x_{2}})^{\frac{1}{4}}\\  &\qquad\qquad\qquad\qquad \sum_{j=3}^{5}b_{j}(x_{3},x_{1},A_{x_{3}})^{\frac{1}{4}}\sum_{j=3}^{5}b_{j}(x_{3},x_{2},A_{x_{3}})^{\frac{1}{4}}\Big)\lambda^{3}(d(x_{1},x_{2},x_{3}))\Big)^{\frac{1}{2}},\\
    \gamma_{2}'&:=\frac{1}{\Var F}\Big(\int  \sum_{j=3}^{5}b_{j}(x_{3},x_{1},A_{x_{3}})\sum_{j=3}^{5}b_{j}(x_{3},x_{2},A_{x_{3}})\lambda^{3}(d(x_{1},x_{2},x_{3}))\Big)^{\frac{1}{2}},\\
    \gamma_{3}'&:=\frac{1}{(\Var F)^{\frac{3}{2}}}\int \sum_{j=1}^{2}b_{j}(x,A_{x})^{\frac{3}{4}}\lambda (dx),\\
    \gamma_{4}'&:=\frac{\int \sum_{j=1}^{2}b_{j}(x,A_{x})^{\frac{3}{4}}\lambda (dx)}{(\Var F)^{2}}\Big(\Big(\int \sum_{j=1}^{2}b_{j}(x,A_{x})^{\frac{1}{2}}\lambda (dx)\Big)^{\frac{1}{2}}+\Big(\int \sum_{j=1}^{2}b_{j}(x,A_{x})\lambda (dx)\Big)^{\frac{1}{4}}\\&\qquad\qquad\qquad\qquad\qquad\qquad\qquad+(\Var F)^{\frac{1}{2}}\Big),\\
    \gamma_{5}'&:=\frac{1}{\Var F}\Big(\int \sum_{j=1}^{2}b_{j}(x,A_{x})\lambda (dx)\Big)^{\frac{1}{2}},\\
    \gamma_{6}'&:=\frac{1}{\Var F}\Big(\int \sum_{j=1}^{2}b_{j}(x_{1},A_{x_{1}})^{\frac{1}{2}}\sum_{j=3}^{5}b_{j}(x_{1},x_{2},A_{x_{1}})^{\frac{1}{2}}+\sum_{j=3}^{5}b_{j}(x_{1},x_{2},A_{x_{1}})\lambda^{2}(d(x_{1},x_{2}))\Big)^{\frac{1}{2}}.
\end{align*}
\end{Theorem}

\begin{Remark}\label{rem3.1} We make some remarks about the general Theorem \ref{general} as follows.
\begin{itemize}
    \item[(i)] The above Theorem \ref{general} generalizes \cite[Theorem 1.2]{last2016normal} by introducing the flexible cost function $D_{x}F_{s}(A_{x})$.
    \item[(ii)] Theorem \ref{general} is valid for deriving normal approximation rates for general functionals or stabilizing functionals not having a known tail probability bound; see Section \ref{sec:mst}. When such tail bounds are known, a more refined result is available (see Corollary~\ref{MTPoi} below).
\end{itemize}
\end{Remark}

When we set $A_{x}=\mbb{X}$ and $\lambda=s\mbb{Q}$, Assumptions \ref{decay}, \ref{kexp} and \ref{fmc} could be leveraged to give upper bounds for
the following crucial probabilities that appear \emph{implicitly} in the proof of Theorem~\ref{general}:
\begin{align}
\label{IJ1}
I_{s}(x)&\coloneqq \mathbb{P}(D_{x}F_{s}\neq 0),\\
J_{s}(x_{1},x_{2})&\coloneqq\mathbb{P}(|D_{x_{1}}F_{s}-D_{x}F_{s}^{x_{2}}|\neq 0),\label{IJ3}
\end{align}
resulting in the following corollary. 

\begin{Corollary}\label{MTPoi}
	Suppose $F_{s}\in L_{\mcal{P}_{s}}^{2}$ and that $F_{s}$ is strongly stabilizing with the radius of stabilization $R_{x}$ decaying exponentially (Assumption~\ref{decay}). Further suppose its cost function satisfies the $\mbb{K}$-exponential bound (Assumption~\ref{kexp}) and the bounded moment condition for $p>4$ (Assumption~\ref{fmc}). Then, there exists a constant $C_{0}>0$ depending only on the constants in \eqref{assQ} and \eqref{H2} such that for $s\ge 1$,
	\begin{equation*}
	d_{K}\left(\frac{F_{s}-\mbb{E}F_{s}}{\sqrt{\Var F_{s}}},N\right)\le C_{0}\left(\frac{\Theta^{\frac{1}{2}}_{\mbb{K},s}}{\Var F_{s}}+\frac{\Theta_{\mbb{K},s}}{(\Var F_{s})^{\frac{3}{2}}}+\frac{\Theta_{\mbb{K},s}^{\frac{5}{4}}+\Theta_{\mbb{K},s}^{\frac{3}{2}}}{(\Var F_{s})^{2}}\right),
	\end{equation*}
	where 
	\begin{align}\label{eq:Thetadef}
	    \Theta_{\mbb{K},s}\coloneqq s\int_{\mbb{X}}e^{-C_{2}\tfrac{(p-4)}{4p}\left(\tfrac{d_{s}(x,\mbb{K})}{2}\right)^{C_{3}}}\mbb{Q}(dx).
	\end{align}
\end{Corollary}

\begin{Corollary}
Under the conditions of Corollary \ref{MTPoi}, assume there exists a constant $C>0$ such that
\begin{equation}\label{varss}
\underset{s\ge 1}{\sup}~\frac{\Theta_{\mbb{K},s}}{\Var F_{s}}\le C,
\end{equation} 
then there exists a constant $C_{0}'>0$ depending on $C$ and \eqref{assQ}-\eqref{H2} such that for $s\ge 1$,
\begin{equation}
d_{K}\left(\frac{F_{s}-\mbb{E}F_{s}}{\sqrt{\Var F_{s}}},N\right)\le C_{0}'\frac{1}{\sqrt{\Var F_{s}}}.\label{eq:optp}
\end{equation}
\label{coros}
\end{Corollary}

Next, we introduce the main theorem for binomial point process. The binomial version of Theorem \ref{general} is not immediately known. Indeed, for the Poisson case we leverage Theorem \ref{6gamma} for our proofs. However, due to the fact that there is no nice counterpart of the second-order Poincar\'{e} inequality (see~\cite{last2016normal}) an analogue of Theorem~\ref{6gamma} is not known for the binomial case. On the other hand, we point out that it is possible to obtain a similar result based on \cite[Theorem 5.1]{lachieze2017new}, which serves as a counterpart of Theorem \ref{6gamma} for binomial setting. Based on this approach, we now present our result for the Binomial setting.

\begin{Theorem}[Normal Approximation for Functionals of Binomial Point Process]\label{MTbino}
	Suppose $F_{n}\in L_{\xi_{n}}^{2}$ and invoke the binomial version of  assumptions in Corollary~\ref{MTPoi}. Then, there exists a constant $C_{0}>0$ depending only on the constants in \eqref{assQ}-\eqref{H2} such that for $n\ge 2$,
	\begin{equation}\label{eq:binomialmain}
	d_{K}\left(\frac{F_{n}-\mbb{E}F_{n}}{\sqrt{\Var F_{n}}},N\right)\le C_{0}\left(\frac{\Theta^{\frac{1}{2}}_{\mbb{K},n}}{\Var F_{n}}+\frac{\Theta_{\mbb{K},n}}{(\Var F_{n})^{\frac{3}{2}}}+\frac{\Theta_{\mbb{K},n}+\Theta_{\mbb{K},n}^{\frac{3}{2}}}{(\Var F_{n})^{2}}\right),
	\end{equation}
	where 
	\begin{align*}
	\Theta_{\mbb{K},n}:=n\int_{\mbb{X}}e^{-C_{2}\tfrac{(p-4)}{4p}\left(\tfrac{d_{n}(x,\mbb{K})}{2}\right)^{C_{3}}}\mbb{Q}(dx).
	\end{align*}
\end{Theorem}

\begin{Remark} We now make the following remark on Theorem \ref{MTbino}. Compared to the Poisson case, the exponent of $\Theta_{\mbb{K},n}$ in the third component of the sum on the right hand side of~\eqref{eq:binomialmain} is different. In essence, this difference can be traced back to a fundamental fact that there is no nice counterpart of the second-order Poincar\'{e} inequality (see~\cite{last2016normal}) for binomial case. Instead, we use the approach taken in~\cite[Theorem 4.2]{lachieze2019normal} to prove Theorem \ref{MTbino}. 
\end{Remark}

\begin{Corollary}\label{corobino}
	Under the conditions of Theorem \ref{MTbino}, assume there exists a constant $C>0$ such that
	\begin{equation}\label{varnn}
	\underset{n\ge 1}{\sup}~\frac{\Theta_{\mbb{K},n}}{\Var F_{n}}\le C,
	\end{equation} 
	then there exists a constant $C_{0}'>0$ depending on $C$ and \eqref{assQ}-\eqref{H2} such that for $n\ge 2$,
	\begin{equation}
	d_{K}\left(\frac{F_{n}-\mbb{E}F_{n}}{\sqrt{\Var F_{n}}},N\right)\le C_{0}'\frac{1}{\sqrt{\Var F_{n}}}.\label{eq:optb}
	\end{equation}
\end{Corollary}

\begin{Remark}
We make the following remarks about the above results on both Poisson and binomial cases. 
\begin{itemize}
\item[(i)] If $\mbb{K}=\mbb{X}$ and $\mbb{Q}(\mbb{X})<\infty$ or $\mbb{Q}$ has a bounded density with respect to the Lebesgue measure on a compact set, the conditions \eqref{varss} and \eqref{varnn} can be simplified as 
\begin{equation}
\underset{s\ge 1}{\sup}~\frac{s}{\Var F_{s}}\le C,\label{vars}
\end{equation}
\begin{equation}
\underset{n\ge 1}{\sup}~\frac{n}{\Var F_{n}}\le C,\label{varn}
\end{equation}
\item[(ii)] \emph{Optimality:} Following~\cite{lachieze2019normal}, we refer to cases where the bounds in~\eqref{eq:optp} and~\eqref{eq:optb} can be attained, as being presumably optimal. Indeed, Corollary \ref{coros} and \ref{corobino} show  that if the variance of the statistics $F_{s},F_{n}$ are bounded below by $\Theta_{\mbb{K},s},\Theta_{\mbb{K},n}$, respectively, a presumably optimal normal approximation rate is achieved. To give an intuition on why the above situation is referred to as being presumably optimal, note that for the case of sums of i.i.d  random variables, non-trivial i.i.d. random variables can be constructed that achieve the upper bounds of the form in~\eqref{eq:optp} and~\eqref{eq:optb}. Formal lower bounds on the optimality are available for the case of integer-valued statistics in~\cite{englund1981remainder} and~\cite{pekoz2013degree}. Furthermore, in a recent work,~\cite{schulte2021rates} established lower bounds for a large class of statistics. 
\end{itemize}
\end{Remark}

\textbf{Comparison to related works.} We now provide some comparisons to the related work. Firstly, our proof techniques, similar to~\cite{lachieze2019normal}, are based on several central ideas proposed in \cite{last2016normal}. For the case of functionals that are expressible as sums of score functions,~\cite{lachieze2019normal} established presumably optimal bounds under the score-based stabilization assumption, for both the binomial and Poisson cases. While they too use second-order cost operators, our Theorem \ref{general} and Theorem \ref{MTbino} handle a much larger class of functionals in comparison (not necessarily as sums of scores). The work of \cite{lachieze2020quantitative} consider general functionals (not necessarily sums) and work under strong stabilization assumption. However, they only consider the Poisson case. To get explicit bounds (e.g., their Corollary 1.5 and Proposition 1.12), they introduce a specific form of $A_x$ in their proofs and their overall approach results in sub-optimal rates in comparison to our results, Corollary \ref{MTPoi}, and to~\cite{lachieze2019normal} in the case when the functional is expressible as sums of scores.  Our Theorem \ref{general} generalizes \cite[Theorem 1.2]{last2016normal} by introducing the flexible cost function $D_{x}F_{s}(A_{x})$ for general functionals of Poisson point process. The work of~\cite{lachieze2017new} also consider normal approximations of general functions (not necessarily as sums). However, their approach is only valid for the binomial case. Moreover, a further investigation of their main theorem \cite[Theorem 4.2]{lachieze2017new} reveals that instead of introducing stabilization notions, their normal approximation bounds are obtained by computing some quantities (for example,  $T,T'$ in \cite[Section 4]{lachieze2017new}), which are complicated to deal with for some functionals, e.g., Euler characteristic; see \cite[Proof of Theorem 3.2]{krebs2021approximation}.\\

\textbf{Applying our main results.} We conclude this section, with the following three-step procedure illustrating how to apply our main theorems, Theorem \ref{general} and Theorem \ref{MTbino}.
\begin{itemize}[noitemsep]
    \item \textit{Step 1:} Check if the functional $F$ is strongly stabilization (i.e., Definition \ref{strongs}), if the tail probability of the radius of stabilization (Assumptions \ref{decay}) could be computed, and verify Assumption~\ref{kexp} on the cost functions.
\begin{itemize}[noitemsep]
\item    If the functional is not strongly stabilizing or no upper bound of the radius of stabilization $R_{x}$ is known, consider the flexible cost functions $D_{x}F_{s}(A_{x})$ with appropriate choice of $A_{x}$ and apply Theorem \ref{general}. 
\end{itemize}
    \item \textit{Step 2:} Check bounded moment condition, i.e., Assumption \ref{fmc}.
    \item \textit{Step 3:} In order to check for presumable optimality, one can seek to bound the variance, i.e., \eqref{vars} and \eqref{varn}.
\end{itemize}
If the above three steps are satisfied, apply Corollary \ref{MTPoi} and Theorem \ref{MTbino} for the Poisson and binomial settings respectively.

\section{Applications}\label{sec:app}
In this section, we illustrate the applicability of our bounds in Theorem \ref{general}, Corollary \ref{MTPoi} and Theorem \ref{MTbino} on several geometric and topological statistics. 

\subsection{Total Edge Length of $k$-Nearest Neighbor Graphs}\label{sec:knnel}
Recall the definition of $k$-NN Graphs in Example~\ref{eg:knn} and define the total edge length of a $k$-NN graph as:
\begin{equation}\label{FKNN}
F^{k\text{-NN}}_s(\mcal{P}_{s}) \coloneqq  \sum_{x\in \mcal{P}_{s}}f_{s}(x,\mcal{P}_{s}),
\end{equation}
with $f_s$ as defined in~\eqref{eq:fsforknn}. similarly, we define $F_{n}^{k\text{-NN}}$ for an underlying binomial point process.

\begin{Theorem}	\label{totalknn}
Assume there exists a constant $c>0$ such that, for $r\le \text{diam}(\mbb{X}) < \infty$, 
\begin{equation}
\underset{x\in\mbb{X}}{\inf}~\mbb{Q}(B_{x}(r))\ge cr^{\omega}. \label{lowerQ}
\end{equation}
If there exists a constant $C>0$ such that
\begin{equation}\label{convars}
\underset{s\ge 1}{\sup}~\frac{s}{\Var F^{k\text{-NN}}_s(\mcal{P}_{s})}\le C,
\end{equation}
then there exists a constant $C_{0}>0$ such that for $s\ge 1$,
	\begin{equation*}
	d_{K}\left(\frac{F^{k\text{-NN}}_s(\mcal{P}_{s})-\mathbb{E}F^{k\text{-NN}}_s(\mcal{P}_{s})}{\sqrt{\mathrm{Var}F^{k\text{-NN}}_s(\mcal{P}_{s})}},N\right)\le C_{0}\frac{1}{\sqrt{s}}.
	\end{equation*}
	And if there exists a constant $C>0$ such that
\begin{equation}\label{convarn}
\underset{n\ge 1}{\sup}~\frac{n}{\Var F^{k\text{-NN}}_n(\xi_{n})}\le C,
\end{equation}
then for $n\ge 2$,
	\begin{equation*}
	d_{K}\left(\frac{F^{k\text{-NN}}_n(\xi_{n})-\mathbb{E}F^{k\text{-NN}}_n(\xi_{n})}{\sqrt{\mathrm{Var}F^{k\text{-NN}}_n(\xi_{n})}},N\right)\le C_{0}\frac{1}{\sqrt{n}}.
	\end{equation*}
\end{Theorem}

\begin{Remark} \label{rem:knn}
We make the following remarks about the above result.
\begin{itemize}
    \item[(i)] Condition~\eqref{lowerQ}, is required in addition to the~\eqref{assQ} for the $k$-NN statistic; see~\cite{lachieze2019normal} for details. Note that the total edge length of a $k$-NN graph \eqref{FKNN} is expressible as a sum of score functions. Hence, the results in~\cite{lachieze2019normal} already provide presumably optimal bounds. Our results above also recover the same bounds. 
    \item[(ii)] Now we compare our results to~\cite{lachieze2020quantitative} in the Poisson setting. Recall that similar to our work, they considered general functionals (not necessarily expressible as sums of scores). However, their generality comes at the cost of not having presumably optimal bounds in the setting of the total edge length of a $k$-NN graph. Specifically, \cite[Proposition 1.12]{lachieze2020quantitative}, term $\sqrt{{b_{n}}/{n}}$ with $b_{n}\rightarrow \infty$ implies that it has a slower rate than ${1}/{\sqrt{n}}$. This highlights the benefit of our approach: despite its generality, we still obtain presumably optimal bounds for this specific special case.
    \item[(iii)] For the binomial setting,~\cite{lachieze2017new} obtained rates in the Kolmogorov metric for the same statistic. However, as discussed in~\cite[Remark (i) below Theorem 3.1]{lachieze2019normal}, their results are sub-optimal and involve additional logarithmic factors, that we avoid. 
    \item[(iv)] When we consider $\mbb{X}$ as a full-dimensional  compact convex subset of $\mbb{R}^{d}$, $\omega=d$, as shown in \cite[Proof of Theorem 6.1]{penrose2001central}\footnote{\cite{penrose2001central} consider the case of $\vartheta=1$. However, a closer examination of the proof shows that it can be easily extended for any $\vartheta>0$.}, the conditions~\eqref{convars} and \eqref{convarn} are satisfied.
    \end{itemize}
  
\end{Remark}

\subsection{Shannon Entropy}\label{Shannon app}
Given an i.i.d. sample $X_{1},X_{2},...,X_{n}$ from a density $q$ on $\mbb{R}^{d}$, the differential (Shannon) entropy is defined as  $H(q):=-\mbb{E}_{X\sim q}\log q(X)=-\int_{\mathbb{R}^{d}}q(x)\log q(x)dx$. The nearest neighbor entropy estimate, also known as the Kozachenko-Leonenko estimator, was first proposed in~\cite{kozachenko1987sample} based on the $1$-NN density estimator. A generalization of this estimator based on k-NN density estimator is given by
\begin{equation*}
\frac{1}{n}\sum_{i=1}^{n}\log\left(\frac{(n-1)V_{d}\rho_{k,i}^{d}}{e^{\Psi(k)}}\right),
\end{equation*}
where $\rho_{k,i}$ is the distance between $X_{i}$ and its $k$-nearest neighbor among $X_{1},X_{2},...,X_{n}$, $V_{d}:=\pi^{\frac{d}{2}}/\Gamma(1+\frac{d}{2})$ is the volume of a unit $d$-dimensional Euclidean ball, $\Psi(k)=-\gamma+\sum_{i=1}^{k-1}1/i$ is the digamma function and $\gamma$ is the Euler-Mascheroni constant \cite{penrose2013limit,berrett2019efficient}. 

The consistency and CLT for the above estimator in a manifold setting were shown in \cite{penrose2013limit} by stabilization theory. However, a non-trivial bias term arises for $d\ge 4$, rendering the above estimator asymptotically inefficient (in the sense of \cite[page 367]{van2000asymptotic}). To have an (asymptotically) unbiased and efficient estimator, the following weighted $k$-NN estimator was proposed in \cite{berrett2019efficient}. Defining $\xi_{n}$ as the binomial point process associated with $X_{1},X_{2},...,X_{n}$ the proposed estimator could be viewed as a functional of $\xi_n$, and is given by
\begin{align*}
F_{n}^{\text{SE}}(\xi_{n}):=\sum_{i=1}^{n}f_{n}^{w}(X_{i},\xi_{n})\quad \text{where}\quad f_{n}^{w}(X_{i},\xi_{n}):=\frac{1}{n}\sum_{j=1}^{k}w_{j}\log\left(\frac{(n-1)V_{d}\rho_{j,i}^{d}}{e^{\Psi(j)}}\right),\label{weighted}
\end{align*}
and $w_{j}$ are the weights (such that $\sum_{j=1}^{n}w_{j}=1$) that are chosen to cancel the dominant bias term and make $F_{n}^{\text{SE}}$ asymptotically efficient. We now provide our normal approximation results for the above estimator, based on a slightly modified set of assumptions considered in~\cite{berrett2019efficient}.

\begin{Theorem}
Consider a density q supported on a compact set $\mbb{X}\subset\mbb{R}^{d}$ with respect to the Lebesgue measure $\mbb{Q}$ in $\mbb{R}^{d}$. 
Let $\mathcal{A}$ denote the class of all decreasing functions $a:(0,\infty)\rightarrow [1,\infty)$ such that $a(\delta)=o(\delta^{-\epsilon})\ \text{as}\ \delta\searrow 0,$ for every $\epsilon>0$. For $a\in\mcal{A}$, let $q$ be $m\coloneqq \lceil \beta \rceil-1$ times differentiable (for $\beta$>0). For $x\in\mbb{X}$, let $r_{a}(x):=(8d^{\frac{1}{2}}a(q(x)))^{-\frac{1}{\beta\wedge 1}}$ and define:
\begin{equation*}
M_{q,a,\beta}(x):=\max\left\{\underset{t=1,...,m}{\max}~\frac{\|q^{(t)}(x)\|}{q(x)},\underset{y\in B_{x}^{o}(r_{a}(x))}{\sup}~\frac{\|q^{(m)}(y)-q^{(m)}(x)\|}{q(x)\|y-x\|^{\beta-m}}\right\},
\end{equation*}
where $B_{x}^{o}(r):=B_{x}(r)\backslash\{x\}$.
Let the density $q$ also satisfy that following conditions:
\begin{equation*}
    \|q\|_{\infty}\le \gamma, \underset{x:q(x)\ge \delta}{\sup}~M_{q,a,\beta}(x)\le a(\delta),\ \forall \delta>0.
\end{equation*}
Define the class of weights as follows: for $k\in\mbb{N}$, let
\begin{align}\label{w}
 \resizebox{0.99\hsize}{!}{$\mcal{W}^{k}:=\bigg\{w\in\mbb{R}^{k}:\sum_{j=1}^{k}w_{j}\frac{\Gamma(j+\frac{2l}{d})}{\Gamma(j)}=0,\ \text{for}\ l=1,...,\lfloor\frac{d}{4}\rfloor, \sum_{j=1}^{k}w_{j}=1, w_{j}=0,\ \text{if}\ j\notin\left\{\lfloor\frac{k}{d}\rfloor,\lfloor\frac{2k}{d}\rfloor,...,k\right\}\bigg\}.$}
\end{align}
Then under the conditions of~\cite[Theorem 1]{berrett2019efficient}, that is, for any $\alpha>d$, $\beta>\frac{d}{2}$ and for any two deterministic sequences of positive integers $k_{0,n}^{*}=k_{0}^{*}$, $k_{1,n}^{*}=k_{1}^{*}$  with $k_{0}^{*}\le k_{1}^{*}$, $k_{0}^{*}/\log^{5}n\rightarrow \infty$, $k_{1}^{*}=O(n^{\tau_{1}})$ and $k_{1}^{*}=o(n^{\tau_{2}})$, where, with $\beta^{*}:=\beta\wedge 1$,
\begin{align*}
	\tau_{1}<\min\left\{\frac{2\alpha}{5\alpha+3d},\frac{\alpha-d}{2\alpha},\frac{4\beta^{*}}{4\beta^{*}+3d}\right\},\quad \text{and}\quad	\tau_{2}:=\min\left\{1-\frac{\frac{d}{4}}{1+\lfloor\frac{d}{4}\rfloor},1-\frac{d}{2\beta}\right\},
\end{align*}
as well as the assumption \eqref{lowerQ}, there exits a constant $C_{0}>0$ (independent of $k,n$) such that
\begin{align*}
d_{K}\left(\frac{F_{n}^{\text{SE}}(\xi_{n})-H(q)}{\sqrt{\mathrm{Var}F_{n}^{\text{SE}}(\xi_{n})}},N\right)\le C_{0}\sqrt{\frac{k}{n}},
\end{align*}
for $k_{0}^{*}\le k\le k_{1}^{*}$.
\label{weightedknn}
\end{Theorem}

\begin{Remark}
We make the following remarks regarding the above result.
\begin{itemize}
\item[(i)] Asymptotic limit theorems for estimators of the Shannon entropy have been obtained, for example, in \cite{penrose2013limit}, \cite{berrett2019efficient}. The result in~\cite{penrose2013limit} is a non-central limit theorem, as their estimator suffers form bias in higher dimensions. The result in \cite{berrett2019efficient}, is a  central limit theorem, which was established under the case that the density is supported on $\mathbb{R}^d$. However, no normal convergence rate results were provided in the above works. To our best knowledge, the above result, is the first normal convergence rate result with the true center $H(q)$.
\item[(ii)]  It is also possible to obtain a similar result using the method in \cite{lachieze2019normal} since the estimator $F_{n}^{\text{SE}}(\xi_{n})$ is expressible as a sum of score functions. 

\item[(iii)] Furthermore, the result in Theorem~\ref{weightedknn} is provided for the binomial case. The asymptotic unbiasedness and efficiency of the weighted $k$-NN estimator of the Shannon entropy based on Poisson point process is open, to the best of our knowledge. 
\end{itemize}
\end{Remark}

\subsection{Euler Characteristic}\label{ECsetting}

Recall Example \ref{egec}. Following the setting of \cite{krebs2021approximation}, consider a bounded density $q$ on $[0,1]^{d}$. Let $\xi_{n}$ be a binomial point process associated with $n$ i.i.d. samples according to the density $q$ and let $\mcal{P}_{n}$ be a Poisson point process with intensity measure $n\mbb{Q}$, where $\mbb{Q}$ has a density $q$ with respect to the Lebesgue measure, i.e., we set $s=n$.

Construct the \v{C}ech complex or the Vietoris-Rips complex $K_{r}(n^{\frac{1}{d}}\mcal{P}_{n}),K_{r}(n^{\frac{1}{d}}\xi_{n})$, see Definition \ref{VR} and \ref{Cech} based on the Poisson point process $\mcal{P}_{n}$ and the binomial point process $\xi_{n}$ respectively with $r>0$ as the filtration time. Here, $K_{r}$ represents both complexes for simplicity. The factor $n^{\frac{1}{d}}$ corresponds to the thermodynamic/critical regime \cite{goel2019strong,owada2020limit,trinh2017remark} such that this is equivalent to the case $nr_{n}^{d}\rightarrow r\in(0,\infty)$ with $r_{n}$ as the filtration time. With the above constriction, the Euler characteristic is given by
\begin{align*}
F_{n}^{\text{EC}}(\mcal{P}_{n}):=\chi(K_{r}(n^{\frac{1}{d}}\mcal{P}_{n}))~\quad\text{and}\quad F_{n}^{\text{EC}}(\xi_{n}):=\chi(K_{r}(n^{\frac{1}{d}}\xi_{n}))
\end{align*}
where $\chi(K(\eta))$ for a filtration $K$ constructed from a point cloud sampled from a point process $\eta$ is defined in Example \ref{egec}.


\begin{Theorem}\label{thm:eulerc}
Under the above setting, for some $T>0$ such that $0<r\le T$, there exists a constant $C_{0}>0$ such that for $n\ge 1$,
	\begin{equation*}
	d_{K}\left(\frac{F_{n}^{\text{EC}}(\mcal{P}_{n})-\mathbb{E}F_{n}^{\text{EC}}(\mcal{P}_{n})}{\sqrt{\mathrm{Var}F_{n}^{\text{EC}}(\mcal{P}_{n})}},N\right)\le C_{0}\frac{1}{\sqrt{n}}.
	\end{equation*}
	And for $n\ge 2$,
	\begin{equation*}
	d_{K}\left(\frac{F_{n}^{\text{EC}}(\xi_{n})-\mathbb{E}F_{n}^{\text{EC}}(\xi_{n})}{\sqrt{\mathrm{Var}F_{n}^{\text{EC}}(\xi_{n})}},N\right)\le C_{0}\frac{1}{\sqrt{n}}.
	\end{equation*}
\end{Theorem}
\begin{Remark}
We make the following remarks about the above result.
\begin{itemize}
    \item[(1)] CLTs and functional limit theorem for Euler characteristic has been studied in \cite{thomas2021functional} by viewing Euler characteristic as a process indexed by $r$. Normal approximation rate of Euler characteristic under binomial and Poisson sampling was obtained in \cite{krebs2021approximation}, by computing certain geometric quantities appearing in the general result in~\cite{lachieze2017new}. Our flexible stabilization method has advantages of avoiding computing several complicated geometric quantities (as done in \cite[Proof of Theorem 3.2]{krebs2021approximation}). 
    \item[(2)] For the Poisson case, \cite{lachieze2020quantitative} require a specific form of $A_{x}$, rendering their result sub-optimal, i.e., the term $\sqrt{{b_{n}}/{n}}$ in \cite[Proposition 1.12]{lachieze2020quantitative},  leads to a slower rate than ${1}/{\sqrt{n}}$ that we obtain above.
    \item [(3)] While Euler characteristic could also be expressible as a sum of score functions, one could possibly leverage the results of~\cite{lachieze2019normal} to derive normal convergence rate. Our goal in this example is to demonstrate the flexibility of our general result.
\end{itemize}

\end{Remark}

All above applications consider stabilizing statistics when there are known tail bounds for the radius of stabilization $R_{x}$, i.e., quantities \eqref{IJ1} to \eqref{IJ3}. Our Theorem \ref{general}, however, can deal with the case when we do not have immediate bounds for those probabilities based on the flexible cost function $D_{x}F_{s}(A_{x})$. We illustrate the above mentioned idea by the following application concerning the minimal spanning tree.

\subsection{Edge length of the Minimal Spanning Tree}\label{sec:mst}

Consider a finite set $V\subset\mbb{R}^{d}$ (usually it is embedded in an underlying graph $G:=(V,E)$). A minimal spanning tree $T$ of $V$ is a connected graph with the vertex set $V$. Define
\begin{equation*}
    M(V):=\underset{T}{\min}\sum_{e\ \text{is an edge of}\ T}|e|,
\end{equation*}
where the minimal is taken over all possible minimal spanning tree T of $V$. According to \cite{penrose2005multivariate} and \cite{chatterjee2017minimal}, the total edge length statistic $M(V)$ does satisfy certain required stabilization properties. However, to the best of our knowledge, there is no result on the rates of  stabilization including quantitative bound on the tail probability of the radius of stabilization. This results in a major difficulty of deriving normal approximation rate for $M(V)$ by some classical methods like \cite{lachieze2019normal}. Our flexible stabilization method can yield the following theorem by picking $A_{x}$ ``strategically'' to make use of some existing bounds.

\begin{Theorem}
    Following the Euclidean setting in \cite{lachieze2020quantitative}, consider $B_{0}$ as the unit hypercube in $\mbb{R}^{d}$ centered at the origin and let $B_{n}:=nB_{0}$, $n\in \mbb{N}_{+}$. Given a homogeneous Poisson process $\mcal{P}(\lambda)$ on $\mbb{R}^{d}$ with intensity $\lambda>0$, let
    \begin{equation*}
    F_{B_{n}}^{\text{MST}}(\mcal{P}(\lambda)):=M(\mcal{P}(\lambda)|_{B_{n}}).
    \end{equation*}
    Then, there exist constants $C_{0}>0, 1>D_{1}>0$ and $D_{2}>0$ not depending on $n$ such that 
    \begin{equation*}
        d_{K}\left(\frac{F_{B_{n}}^{\text{MST}}(\mcal{P}(\lambda))-\mbb{E}F_{B_{n}}^{\text{MST}}(\mcal{P}(\lambda))}{\sqrt{\Var F_{B_{n}}^{\text{MST}}(\mcal{P}(\lambda))}},N\right)\le 
        \left\{
        \begin{aligned}
        &C_{0}n^{-D_{1}}, \quad \text{if}\ d=2,\\
        &C_{0}(\log n)^{-D_{2}}, \quad \text{if}\ d\ge 3.
        \end{aligned}
        \right.
    \end{equation*}
    \label{mst}
\end{Theorem}

\begin{Remark} The normal approximation rate of the edge length statistic of the minimal spanning tree has been derived previously in \cite{chatterjee2017minimal} and in \cite{lachieze2020quantitative} with similar results as Theorem~\ref{mst}. However, \cite{chatterjee2017minimal} only focused on the minimal spanning tree therefore it is hard to generalize for other stabilizing functionals. \cite{lachieze2020quantitative} used the similar idea of introducing the set $A_{x}$ but their bounds usually give sub-optimal normal convergence rates (for e.g., see Remark~\ref{rem:knn} regarding the total edge length of k-nearest neighbor graphs in Section~\ref{sec:knnel}) than ours due to the specific form of their set $A_{x}$ lacking flexibility. 
\end{Remark}

\section{Proofs for Section~\ref{sec:main}}\label{proofs}
\subsection{Proof of Theorem~\ref{general}}
We first prove our main theorem for Poisson case, Theorem \ref{general}, based on the flexible cost function $D_{x}F(A_{x})$. Without loss of generality, all the constant $C>0$ mentioned in this section refer to universal constants that might take different values in each step. Our key tool for proving Theorem~\ref{general} is based on \cite[Theorem 1.2]{last2016normal}, which we restate below. 
\begin{Theorem}[\cite{last2016normal}]
	Let $F$ be a measurable functional satisfying the following two conditions: 
	\begin{align*}
	\mathbb{E}F(\mcal{P}(\lambda))^{2}<\infty \qquad\text{and}\qquad\mathbb{E}\int(D_{x}F)^{2}\lambda(dx)<\infty.
	\end{align*}
	Let $N$ be a standard normal random variable. Then,
	\begin{align*}
	d_{K}\left(\frac{F-\mathbb{E}F}{\sqrt{\Var F}},N\right)\le \sum_{i=1}^6\gamma_{i},
	\end{align*}
	where
\begin{align*}
\gamma_{1}&\coloneqq\frac{4}{\Var F}\left(\int (\mathbb{E}(D_{x_{1}}F)^{2}(D_{x_{2}}F)^{2})^{\frac{1}{2}} (\mathbb{E}(D_{x_{1},x_{3}}F)^{2}(D_{x_{2},x_{3}}F)^{2})^{\frac{1}{2}}\lambda^{3}(d(x_{1},x_{2},x_{3}))\right)^{\frac{1}{2}},\\
\gamma_{2}&\coloneqq\frac{1}{\Var F}\left(\int  (\mathbb{E}(D_{x_{1},x_{3}}F)^{2}(D_{x_{2},x_{3}}F)^{2})^{2}\lambda^{3}(d(x_{1},x_{2},x_{3}))\right)^{\frac{1}{2}},\\
\gamma_{3}&\coloneqq\frac{1}{(\Var F)^{\frac{3}{2}}}\int \mathbb{E}|D_{x}F|^{3}\lambda (dx),\\
\gamma_{4}&\coloneqq\frac{1}{2(\Var F)^{2}}(\mathbb{E}(F-\mathbb{E}F)^{4})^{\frac{1}{4}}\int (\mathbb{E}(D_{x}F)^{4})^\frac{3}{4}\lambda (dx),\\
\gamma_{5}&\coloneqq\frac{1}{\Var F}\left(\int \mathbb{E}(D_{x}F)^{4}\lambda (dx)\right)^{\frac{1}{2}},\\
\gamma_{6}&\coloneqq\frac{1}{\Var F}\left(\int 6(\mathbb{E}(D_{x_{1}}F)^{4})^{\frac{1}{2}}\left(\mathbb{E}(D_{x_{1},x_{2}}F)^{4}\right)^{\frac{1}{2}}+3\mathbb{E}(D_{x_{1},x_{2}}F)^{4} \lambda^{2}(d(x_{1},x_{2}))\right)^{\frac{1}{2}}.
\end{align*}
\label{6gamma}
\end{Theorem}

\begin{proof}[Proof of Theorem \ref{general}]
The idea is to use the above theorem, and to bound the first and second order cost functions appearing in the quantities $\gamma_i, i = 1,\ldots,6$ by using our flexible approach. To this end, we rewrite the first and second order cost functions as:
\begin{align*}
    D_{x}F_{s}=(D_{x}F_{s}-D_{x}F_{s}(A_{x}))+D_{x}F_{s}(A_{x}),
\end{align*}
and
\begin{align*}
 \resizebox{0.99\hsize}{!}{$
    D_{x_{1},x_{3}}F_{s}=(D_{x_{3}}F_{s}^{x_{1}}-D_{x_{3}}F_{s}^{x_{1}}(A_{x_{3}}))+(D_{x_{3}}F_{s}^{x_{1}}(A_{x_{3}})-D_{x_{3}}F_{s}(A_{x_{3}}))+(D_{x_{3}}F_{s}(A_{x_{3}})-D_{x_{3}}F_{s}).$}
\end{align*}
 We start with $\gamma_3$. By the fact that $(a+b)^{3}\le 4(a^{3}+b^{3})$ for any $a\ge 0,b\ge 0$, we have
\begin{align*}
(\Var F)^{\frac{3}{2}}\gamma_{3}&=\int \mathbb{E}|(D_{x}F-D_{x}F(A_{x}))+D_{x}F(A_{x})|^{3}\lambda (dx)\\
&\le \int 4\mathbb{E}|D_{x}F-D_{x}F(A_{x})|^{3}\lambda (dx)+\int 4\mathbb{E}|D_{x}F(A_{x})|^{3}\lambda (dx).
\end{align*}
By H\"{o}lder's inequality and the assumptions in Theorem \ref{general}, we then have 
\begin{align}\label{3F1}
\mathbb{E}|D_{x}F(A_{x})|^{3}\le (\mathbb{E}|D_{x}F(A_{x})|^{4})^{\frac{3}{4}}1^{1-\frac{3}{4}}\le b_{2}(x,A_{x})^{\frac{3}{4}}.
\end{align}
Similarly,
\begin{align}\label{3F2}
\mathbb{E}|D_{x}F-D_{x}F(A_{x})|^{3}\le b_{1}(x,A_{x})^{\frac{3}{4}}.
\end{align}
Therefore,
\begin{equation*}
\begin{aligned}
(\Var F)^{\frac{3}{2}}\gamma_{3}\le C\int \sum_{j=1}^{2}b_{j}(x,A_{x})^{\frac{3}{4}}\lambda(dx)\coloneqq C (\Var F)^{\frac{3}{2}}\gamma'_{3}.
\end{aligned}
\end{equation*}

Next, we turn to $\gamma_4$. According to \cite[Lemma 4.3]{last2016normal}, we have 
\begin{align*}
\mathbb{E}(F-\mathbb{E}F)^{4}\le \max\left\{256\left(\int (\mathbb{E}(D_{x}F)^{4})^{\frac{1}{2}}\lambda (dx)\right)^{2},4\int \mathbb{E}(D_{x}F)^{4}\lambda (dx)+2(\Var F)^{2}\right\}.
\end{align*}
Consequently,
\begin{align*}
\resizebox{0.98\hsize}{!}{$
\gamma_{4}\le \frac{\int(\mathbb{E}(D_{x}F)^{4})^{\frac{3}{4}}\lambda (dx)}{2(\Var F)^{2}}\left(4\left(\int (\mathbb{E}(D_{x}F)^{4}\right)^{\frac{1}{2}}\lambda (dx))^{\frac{1}{2}}+\sqrt{2}\left(\int \mathbb{E}(D_{x}F\right)^{4}\lambda (dx))^{\frac{1}{4}}+2^{\frac{1}{4}}(\Var F)^{\frac{1}{2}}\right).
$}
\end{align*}
By calculations similar to $\gamma_{3}$, we have 
\begin{align*}
\int(\mathbb{E}(D_{x}F)^{4})^{\frac{1}{2}}\lambda (dx)\le C\int \sum_{j=1}^{2}b_{j}(x,A_{x})^{\frac{1}{2}}\lambda(dx).
\end{align*}
Similarly,
\begin{align*}
\int (\mathbb{E}(D_{x}F)^{4})^{\frac{3}{4}}\lambda (dx)&\le C\int \sum_{j=1}^{2}b_{j}(x,A_{x})^{\frac{3}{4}}\lambda(dx),\\
\int \mathbb{E}(D_{x}F_{s})^{4}\lambda (dx)&\le C\int \sum_{j=1}^{2}b_{j}(x,A_{x})\lambda(dx).
\end{align*}
Combining all above, we have
\begin{align*}
    \gamma_{4}\le \frac{C\int \sum_{j=1}^{2}b_{j}(x,A_{x})^{\frac{3}{4}}\lambda (dx)}{(\Var F)^{2}}\bigg(\bigg(&\int \sum_{j=1}^{2}b_{j}(x,A_{x})^{\frac{1}{2}}\lambda (dx)\bigg)^{\frac{1}{2}}\\&+\bigg(\int \sum_{j=1}^{2}b_{j}(x,A_{x})\lambda (dx)\bigg)^{\frac{1}{4}}+(\Var F)^{\frac{1}{2}}\bigg):=C \gamma_{4}'.
\end{align*}
Furthermore, by similar arguments on bounding $\gamma_{3}$, it leads to  
\begin{equation*}
\gamma_{5}\le \frac{C}{\Var F}\left(\int \sum_{j=1}^{2}b_{j}(x,A_{x})\lambda(dx)\right)^{\frac{1}{2}}:=C\gamma_{5}'.
\end{equation*}

We next move on to bounding the remaining part: $\gamma_{2}$, $\gamma_{2}$ and $\gamma_{6}$, that involve the second order cost functions. Similarly to the first order cost function, we have
\begin{align}\label{3S}
\mathbb{E}(D_{x_{1},x_{2}}F)^{4}\le C \sum_{j=3}^{5}b_{j}(x_{1},x_{2},A_{x_{1}}).
\end{align}
Using Cauchy-Schwarz inequality, we have
\begin{align}\label{CS1}
(\mathbb{E}(D_{x_{1}}F)^{2}(D_{x_{2}}F)^{2})^{\frac{1}{2}}&\le \mathbb{E}(D_{x_{1}}F)^{4})^{\frac{1}{4}}\mathbb{E}(D_{x_{2}}F)^{4})^{\frac{1}{4}},\\
(\mathbb{E}(D_{x_{1},x_{3}}F)^{2}(D_{x_{2},x_{3}}F)^{2})^{\frac{1}{2}}&\le \mathbb{E}(D_{x_{1},x_{3}}F)^{4})^{\frac{1}{4}}\mathbb{E}(D_{x_{2},x_{3}}F)^{4})^{\frac{1}{4}}\label{CS2}.
\end{align}
With all results above, according to \eqref{3F1} to \eqref{CS2}, we give an upper bound for $\gamma_{1}$ in a similar way by H\"{o}lder's inequality:
\begin{align*}
\gamma_{1}&\le \frac{C}{\Var F} \bigg(\int \bigg(\sum_{j=1}^{2}b_{j}(x_{1},A_{x_{1}})^{\frac{1}{4}}\sum_{j=1}^{2}b_{j}(x_{2},A_{x_{2}})^{\frac{1}{4}}\\  &\qquad\qquad\qquad\qquad \sum_{j=3}^{5}b_{j}(x_{3},x_{1},A_{x_{3}})^{\frac{1}{4}}\sum_{j=3}^{5}b_{j}(x_{3},x_{2},A_{x_{3}})^{\frac{1}{4}}\bigg)\lambda^{3}(d(x_{1},x_{2},x_{3}))\bigg)^{\frac{1}{2}}\\&:=C\gamma_{1}'.
\end{align*}
Similarly, we have 
\begin{align*}
    \gamma_{2}&\le \frac{C}{\Var F}\left(\int  \sum_{j=3}^{5}b_{j}(x_{3},x_{1},A_{x_{3}})\sum_{j=3}^{5}b_{j}(x_{3},x_{2},A_{x_{3}})\lambda^{3}(d(x_{1},x_{2},x_{3}))\right)^{\frac{1}{2}}\\&:=C\gamma_{2}',\\
    \gamma_{6}&\le \frac{C}{\Var F}\bigg(\int \sum_{j=1}^{2}b_{j}(x_{1},A_{x_{1}})^{\frac{1}{2}}\sum_{j=3}^{5}b_{j}(x_{1},x_{2},A_{x_{1}})^{\frac{1}{2}}+\sum_{j=3}^{5}b_{j}(x_{1},x_{2},A_{x_{1}})\lambda^{2}(d(x_{1},x_{2}))\bigg)^{\frac{1}{2}}\\&:=C\gamma_{6}'.
\end{align*}
Combining the obtained bounds above for $\gamma_{i}$, $1\le i\le 6$, we complete the proof.
\end{proof}

\subsection{Proof of Corollary \ref{MTPoi}}
Without loss of generality, we assume $c_{1}=k_{1}:=C_{1}, c_{2}=k_{2}:=C_{2}$ and $c_{3}=k_{3}:=C_{3}$. In Theorem \ref{general}, we set $A_{x}=\mbb{X}$ and $\lambda=s\mbb{Q}$. Then, we can set
\begin{align*}
    b_{1}(x,A_{x})&=0,\\
    b_{3}(x_{1},x_{2},A_{x_{1}})&=0,\\
    b_{4}(x_{1},x_{12},A_{x_{1}})&=0.
\end{align*}
 According to Assumption \ref{fmc}, by H\"{o}lder's inequality, we have:
\begin{align*}
    b_{2}(x,\mbb{X}):=\mathbb{E}|D_{x}F_{s}|^{4}\le \mathbb{E}|D_{x}F_{s}|^{p}\mbb{P}(D_{x}F_{s}\neq 0)^{1-\frac{4}{p}}\le C\mbb{P}(D_{x}F_{s}\neq 0)^{1-\frac{4}{p}},
\end{align*}
and
\begin{align*}
    b_{5}(x_{1},x_{2},\mbb{X}):=\mbb{E}|D_{x_{1}}F_{s}^{x_{2}}-D_{x_{1}}F_{s}|^{4}&\le \mbb{E}|D_{x_{1}}F_{s}^{x_{2}}-D_{x_{1}}F_{s}|^{p}\mbb{P}(|D_{x_{1}}F_{s}^{x_{2}}-D_{x_{1}}F_{s}|\neq 0)^{1-\frac{4}{p}}\\&\le C\mbb{P}(|D_{x_{1}}F_{s}^{x_{2}}-D_{x_{1}}F_{s}|\neq 0)^{1-\frac{4}{p}}.
\end{align*}

Define
	\begin{align*}
	\phi_{s}&\coloneqq s\int_{\mbb{X}}\mathbb{P}(D_{x}F_{s}\neq 0)^{\frac{p-4}{2p}}\mbb{Q}(dx),\\
	\psi_{s}(x_{1},x_{2})&\coloneqq \mathbb{P}(|D_{x_{1}}F_{s}^{x_{2}}-D_{x_{1}}F_{s}|\neq 0)^{\frac{p-4}{4p}}.
	\end{align*}
    With all $b_{i}$, $1\le i\le 5$, given above, we will bound all $\gamma_{i}'$, $1\le i\le 6$ in Theorem \ref{general}. We again start with $\gamma_{3}'$:
    \begin{align*}
    \gamma_{3}'&=\frac{1}{(\Var F_{s}^{\frac{3}{2}}}\int \sum_{j=1}^{2}b_{j}(x,A_{x})^{\frac{3}{4}}\lambda (dx)\\
    &=Cs\int_{\mbb{X}}\mathbb{P}(D_{x}F_{s}\neq 0)^{(1-\frac{4}{p})\frac{3}{4}}\mbb{Q}(dx)\\&\le C\phi_{s}.
    \end{align*}
    Similarly, as for $\gamma_{3}'$, it holds that
    \begin{align*}
        \gamma_{5}'&=\frac{1}{\Var F_{s}}\left(\int \sum_{j=1}^{2}b_{j}(x,A_{x})\lambda(dx)\right)^{\frac{1}{2}}\\&=\frac{1}{\Var F_{s}}\left(s\int_{\mbb{X}}\mathbb{P}(D_{x}F_{s}\neq 0)^{1-\frac{4}{p}}\mbb{Q}(dx)\right)^{\frac{1}{2}}\\&\le C\phi_{s}^{\frac{1}{2}}.
    \end{align*}
    For $\gamma_{4}'$, we have
    \begin{align}\nonumber
        \gamma_{4}'=\frac{\int \sum_{j=1}^{2}b_{j}(x,A_{x})^{\frac{3}{4}}\lambda (dx)}{(\Var F_{s})^{2}}\bigg(\bigg(&\int \sum_{j=1}^{2}b_{j}(x,A_{x})^{\frac{1}{2}}\lambda (dx)\bigg)^{\frac{1}{2}}\\&+\bigg(\int \sum_{j=1}^{2}b_{j}(x,A_{x})\lambda (dx)\bigg)^{\frac{1}{4}}+(\Var F_{s})^{\frac{1}{2}}\bigg)\label{eq:tempgamma4}.
    \end{align}
     Respectively,
     \begin{align*}
         \bigg(\int \sum_{j=1}^{2}b_{j}(x,A_{x})^{\frac{1}{2}}\lambda (dx)\bigg)^{\frac{1}{2}}&=\left(s\int_{\mbb{X}}\mathbb{P}(D_{x}F_{s}\neq 0)^{(1-\frac{4}{p})\frac{1}{2}}\mbb{Q}(dx)\right)^{\frac{1}{2}}\\&\le C\phi_{s}^{\frac{1}{2}},
     \end{align*}
    \begin{align*}
        \bigg(\int \sum_{j=1}^{2}b_{j}(x,A_{x})\lambda (dx)\bigg)^{\frac{1}{4}}&=\left(s\int_{\mbb{X}}\mathbb{P}(D_{x}F_{s}\neq 0)^{1-\frac{4}{p}}\mbb{Q}(dx)\right)^{\frac{1}{4}}\\&\le C\phi_{s}^{\frac{1}{4}},
    \end{align*}
    and
    \begin{align*}
        \int \sum_{j=1}^{2}b_{j}(x,A_{x})^{\frac{3}{4}}\lambda (dx)&=s\int_{\mbb{X}}\mathbb{P}(D_{x}F_{s}\neq 0)^{(1-\frac{4}{p})\frac{3}{4}}\mbb{Q}(dx)\\&\le C\phi_{s}.
    \end{align*}
From the above calculations, we see that the exponent $(p-4)/2p$ is set so that indeed $\phi_s$ provides an upper bound for all the terms appearing in the right hand of~\eqref{eq:tempgamma4}. Hence, we have
\begin{align*}
\gamma_{4}'\le C\left(\frac{\phi_{s}^{\frac{3}{2}}+\phi_{s}^{\frac{5}{4}}}{(\Var F_{s})^{2}}+\frac{\phi_{s}}{(\Var F_{s})^{\frac{3}{2}}}\right).
\end{align*}
For terms that include $b_{j}$, $3\le j\le 5$, we have 
\begin{align*}
\gamma_{1}'&=\frac{1}{\Var F_{s}} \bigg(\int \bigg(\sum_{j=1}^{2}b_{j}(x_{1},A_{x_{1}})^{\frac{1}{4}}\sum_{j=1}^{2}b_{j}(x_{2},A_{x_{2}})^{\frac{1}{4}}\\  &\qquad\qquad\qquad\qquad \sum_{j=3}^{5}b_{j}(x_{3},x_{1},A_{x_{3}})^{\frac{1}{4}}\sum_{j=3}^{5}b_{j}(x_{3},x_{2},A_{x_{3}})^{\frac{1}{4}}\bigg)\lambda^{3}(d(x_{1},x_{2},x_{3}))\bigg)^{\frac{1}{2}}\\&\le C\frac{s^{\frac{3}{2}}}{\Var F_{s}}\Bigg(\int_{\mbb{X}^{3}}\Big(\mathbb{P}(|D_{x_{1}}F_{s}^{x_{2}}-D_{x_{1}}F_{s}|\neq 0)^{\frac{p-4}{4p}}\\&\qquad \qquad\qquad \qquad\mathbb{P}(|D_{x_{2}}F_{s}^{x_{3}}-D_{x_{2}}F_{s}|\neq 0)^{\frac{p-4}{4p}}\Big)\mbb{Q}^{3}(d(x_{1},x_{2},x_{3}))\Bigg)^{\frac{1}{2}}\\&=C\frac{s^{\frac{3}{2}}}{\Var F_{s}}\sqrt{\int_{\mbb{X}}\left(\int_{\mbb{X}}\psi_{s}(x_{1},x_{2})\mbb{Q}(dx_{2})\right)^{2}\mbb{Q}(dx_{1})}.
\end{align*}
Similarly,
\begin{align*}
    \gamma_{2}'&=\frac{C}{\Var F_{s}}\left(\int  \sum_{j=3}^{5}b_{j}(x_{3},x_{1},A_{x_{3}})\sum_{j=3}^{5}b_{j}(x_{3},x_{2},A_{x_{3}})\lambda^{3}(d(x_{1},x_{2},x_{3}))\right)^{\frac{1}{2}}\\& \le C\frac{s^{\frac{3}{2}}}{\Var F_{s}}\Bigg(\int_{\mbb{X}^{3}}\Big(\mathbb{P}(|D_{x_{1}}F_{s}^{x_{2}}-D_{x_{1}}F_{s}|\neq 0)^{\frac{p-4}{p}}\\&\qquad \qquad\qquad \qquad\mathbb{P}(|D_{x_{2}}F_{s}^{x_{3}}-D_{x_{2}}F_{s}|\neq 0)^{\frac{p-4}{p}}\Big)\mbb{Q}^{3}(d(x_{1},x_{2},x_{3}))\Bigg)^{\frac{1}{2}}\\&= C\frac{s^{\frac{3}{2}}}{\Var F_{s}}\sqrt{\int_{\mbb{X}}\left(\int_{\mbb{X}}\psi_{s}(x_{1},x_{2})\mbb{Q}(dx_{2})\right)^{2}\mbb{Q}(dx_{1})},
\end{align*}
and
\begin{align*}
    \gamma_{6}'&=\frac{1}{\Var F_{s}}\bigg(\int \sum_{j=1}^{2}b_{j}(x_{1},A_{x_{1}})^{\frac{1}{2}}\sum_{j=3}^{5}b_{j}(x_{1},x_{2},A_{x_{1}})^{\frac{1}{2}}+\sum_{j=3}^{5}b_{j}(x_{1},x_{2},A_{x_{1}})\lambda^{2}(d(x_{1},x_{2}))\bigg)^{\frac{1}{2}}\\&\le C\frac{s}{\Var F_{s}}\Big(\int_{\mbb{X}^{2}}\mathbb{P}(|D_{x_{2}}F_{s}^{x_{3}}-D_{x_{2}}F_{s}|\neq 0)^{\frac{p-4}{2p}}\\&\qquad \qquad\qquad+\mathbb{P}(|D_{x_{2}}F_{s}^{x_{3}}-D_{x_{2}}F_{s}|\neq 0)^{\frac{p-4}{p}}\mbb{Q}^{2}(d(x_{1},x_{2}))\Big)^{\frac{1}{2}}\\&\le C\frac{s}{\Var F_{s}}\left(\int_{\mbb{X}^{2}}\mathbb{P}(|D_{x_{2}}F_{s}^{x_{3}}-D_{x_{2}}F_{s}|\neq 0)^{\frac{p-4}{2p}}\mbb{Q}^{2}(d(x_{1},x_{2}))\right)^{\frac{1}{2}}\\&\le C\frac{s}{\Var F_{s}}\sqrt{\int_{\mbb{X}^{2}}\psi_{s}(x_{1},x_{2})^{2}\mbb{Q}^{2}(d(x_{1},x_{2}))}.
\end{align*}
Therefore, combining all bounds for $\gamma_{i}'$, $1\le i\le 6$, we have by Theorem \ref{general},
\begin{equation}\label{missingthm}
	d_{K}\left(\frac{F_{s}-\mbb{E}F_{s}}{\sqrt{\Var F_{s}}},N\right)\le C(\theta_{1}+\theta_{2}+\theta_{3}),
 \end{equation}
 where 
	\begin{align}
	\theta_{1}&\coloneqq \frac{s}{\Var F_{s}}\sqrt{\int_{\mbb{X}^{2}}\psi_{s}(x_{1},x_{2})^{2}\mbb{Q}^{2}(d(x_{1},x_{2}))} \label{theta1},\\
    \theta_{2}&\coloneqq\frac{s^{\frac{3}{2}}}{\Var F_{s}}\sqrt{\int_{\mbb{X}}\left(\int_{\mbb{X}}\psi_{s}(x_{1},x_{2})\mbb{Q}(dx_{2})\right)^{2}\mbb{Q}(dx_{1})},\label{theta2}\\
	\theta_{3}&\coloneqq\frac{(\phi_{s})^{\frac{1}{2}}}{\Var F_{s}}+\frac{\phi_{s}}{(\Var F_{s})^{\frac{3}{2}}}+\frac{\phi_{s}^{\frac{5}{4}}+\phi_{s}^{\frac{3}{2}}}{(\Var F_{s})^{2}}.\label{theta3}
	\end{align}

We next proceed to obtain refined bounds for $\phi_s$ and $\psi_{s}(x_{1},x_{2})$. Before proceeding, we recall the following definitions from ~\eqref{IJ1} to \eqref{IJ3}:
\begin{align*}
I_{s}(x)&\coloneqq \mathbb{P}(D_{x}F_{s}\neq 0),\\
J_{s}(x_{1},x_{2})&\coloneqq\mathbb{P}(|D_{x_{1}}F_{s}-D_{x}F_{s}^{x_{2}}|\neq 0).
\end{align*}

\begin{Lemma}\label{lemma61}
	Assume that all conditions in Corollary \ref{MTPoi} hold, and recall that $R_{x}$ denotes the radius of stabilization. Then,
	\begin{align*}
	\phi_{s} \le Cs\int_{\mbb{X}}e^{-C_{2}\tfrac{p-4}{2p}d_{s}(x,\mbb{K})^{C_{3}}}\mbb{Q}(dx),
    \end{align*}
    and
    \begin{align*}
    \psi_{s}(x_{1},x_{2})\le Ce^{-C_{2}\frac{p-4}{4p}d_{s}(x_{1},x_{2})^{C_{3}}}.
	\end{align*}
\end{Lemma}

\begin{proof}[Proof of Lemma~\ref{lemma61}]
    Based on the Assumptions \ref{decay} and \ref{kexp}, we immediately obtain
	\begin{align*}
	I_{s}(x)&\le C_{1}e^{-C_{2}d_{s}(x,\mbb{K})^{C_{3}}},\\ 
	J_{s}(x_{1},x_{2})&\le C_{1}e^{-C_{2}\max\{d_{s}(x_{1},x_{2}),d_{s}(x_{1},\mbb{K}),d_{s}(x_{2},\mbb{K})\}^{C_{3}}}.
	\end{align*}
Therefore, we have
	\begin{align*}
	\phi_{s}& \le Cs\int_{\mbb{X}}e^{-C_{2}\frac{p-4}{2p}d_{s}(x,\mbb{K})^{C_{3}}}\mbb{Q}(dx)
	\end{align*}
Similarly, we obtain the stated upper bound for $\psi_{s}(x_{1},x_{2})$.
\end{proof}

\begin{Lemma}\label{lemma62}
	Suppose the condition \eqref{assQ} holds. Then, for any $x\in\mbb{X}$ and $r\ge 0$,  we have
$$\mbb{Q}(B_{x}(r))\le \kappa r^{\omega}.$$	
\end{Lemma}
\begin{proof}[Proof of Lemma~\ref{lemma62}]
	For any $x\in\mbb{X}$ fixed, consider $Q(r):=\mbb{Q}(B_{x}(r)),r\ge 0$ as an increasing function of $r$. According to Lebesgue's theorem for the differentiability of monotone functions, the derivative $Q'(r)$ exists almost everywhere and then with the condition \eqref{assQ},
	\begin{equation*}
	\mbb{Q}(B_{x}(r))-\mbb{Q}(B_{x}(0))\le \int_{0}^{r}Q'(u)du\le \int_{0}^{r}\kappa \omega u^{\omega-1}du=\kappa r^{\omega}.
	\end{equation*}
	Note that $\mbb{Q}(B_{x}(0))=0$. Therefore, we obtain the desired result.
\end{proof}

\begin{Lemma}
	Suppose the condition \eqref{assQ} holds. For any $x\in\mbb{X}$, $r\ge 0$ and $\alpha>0$, there exists a constant $C>0$ such that
	\begin{equation*}
	\int_{\mbb{X}\backslash B_{x}(r)}e^{-\alpha d_{s}(x,y)^{C_{3}}}\mbb{Q}(dy)\le \frac{C}{s}e^{-\frac{1}{2}\alpha(s^{1/\omega}r)^{C_{3}}}.
	\end{equation*}
	\label{lemma63}
\end{Lemma}

\begin{proof}[Proof of Lemma~\ref{lemma63}]
	Let $\{r_{n}\}_{n=1}^{\infty}$ be an increasing sequence satisfying 
	\begin{align*}
	r_{1}:=r, \qquad \underset{n\rightarrow\infty}{\lim}~r_{n}=\infty, \qquad\text{and}\qquad \underset{n\rightarrow\infty}{\lim}~\underset{n\ge 2}{\sup}~|r_{n}-r_{n-1}|=0.
	\end{align*}
	Then, 
	\begin{align*}
	\int_{\mbb{X}\backslash B_{x}(r)}e^{-\alpha d_{s}(x,y)^{C_{3}}}\mbb{Q}(dy)&\le\sum_{n=2}^{\infty}e^{-\alpha (s^{1/\omega }r_{n-1})^{C_{3}}}\mbb{Q}(B_{x}(r_{n})\backslash B_{x}(r_{n-1}))\\&\le \sum_{n=2}^{\infty}e^{-\alpha (s^{1/\omega}r_{n-1})^{C_{3}}}\kappa\omega r_{n-1}^{\omega-1}(r_{n}-r_{n-1})\\&=\int_{r}^{\infty}e^{-\alpha (s^{1/\omega}u)^{C_{3}}}\kappa\omega u^{\omega-1}du.
	\end{align*}
	Therefore, it suffices to show
	\begin{align*}
	\zeta(s,r):=se^{\frac{\alpha}{2} (s^{1/\omega}r)^{C_{3}}}\int_{r}^{\infty}e^{-\alpha (s^{1/\omega}u)^{C_{3}}}u^{\omega-1}du
	\end{align*}
	is bounded on $\Theta:=\{(s,r):s\ge 1,r\ge 0\}$.
	Since $\zeta(s,r)$ is a continuous function, then we only need to show 
	\begin{align*}
	\underset{(s,r)\rightarrow \partial\Theta}{\lim}\zeta(s,r)<\infty.
	\end{align*}
	Note that
	\begin{equation*}
	\zeta(s,r)\le se^{(\frac{\alpha}{2}-\frac{2\alpha}{3}) (s^{1/\omega}r)^{C_{3}}}\int_{r}^{\infty}e^{-\frac{\alpha}{3} (s^{1/\omega}u)^{C_{3}}}u^{\omega-1}du.
	\end{equation*}
	Moreover, let 
	\begin{equation*}
	\int_{r}^{\infty}e^{-\frac{\alpha}{3} (s^{1/\omega}u)^{C_{3}}}u^{\omega-1}du\le \int_{0}^{\infty}e^{-\frac{\alpha}{3} (s^{1/\omega}u)^{C_{3}}}u^{\omega-1}du:=\eta(s),
	\end{equation*}
	with
	\begin{equation*}
	\underset{s\rightarrow\infty}{\lim}\eta(s)<\infty.
	\end{equation*}
	Then,
	\begin{equation*}
	\zeta(s,r)\le se^{(\frac{\alpha}{2}-\frac{2\alpha}{3}) (s^{1/\omega}r)^{C_{3}}}~ \eta(s).
	\end{equation*}
	Consequently, noting $$\frac{\alpha}{2}-\frac{2\alpha}{3}=-\frac{1}{6}\alpha<0,$$ we have
	\begin{equation*}
	\underset{(s,r)\rightarrow \partial\Theta}{\lim}\zeta(s,r)<\infty,
	\end{equation*}
giving us the desired result.
\end{proof}

\begin{Lemma}
For any $s\ge 1$, $r\ge 0$ and $\beta>0$, there exists a constant $C>0$ such that 
\begin{equation*}
sr^{\omega}e^{-\beta(s^{1/\omega}r)^{C_{3}}}\le C e^{-\frac{1}{2}\beta(s^{1/\omega}r)^{C_{3}}}.
\end{equation*}
\label{lemma64}
\end{Lemma}

\begin{proof}[Proof of Lemma~\ref{lemma64}]
	Let
	\begin{equation*}
	\mu(s,r):=sr^{\omega}e^{-\frac{1}{2}\beta(s^{1/\omega}r)^{C_{3}}}.
	\end{equation*}
    It suffices to prove
	\begin{equation*}
	\underset{s\ge 1,r\ge 0}{\sup}\mu(s,r)<\infty.
	\end{equation*}
	Similar to $\zeta(s,r)$ in Lemma \ref{lemma63}, we only need to show
	\begin{equation*}
	\underset{(s,r)\rightarrow \partial\Theta}{\lim}\mu(s,r)<\infty,
	\end{equation*}
	where $\Theta:=\{(s,r):s\ge 1,r\ge 0\}$.
	We note that the above claim follows by calculations similar to that in the proof of Lemma~\ref{lemma63} and the fact that ${-\beta}/{2}<0$, thus providing the desired result.
\end{proof}

Before proceeding, we recall the definition of $\Theta_{\mbb{K},s}$ from~\eqref{eq:Thetadef} for convenience:
\begin{align*}
	\Theta_{\mbb{K},s}\coloneqq s\int_{\mbb{X}}e^{-C_{2}\tfrac{(p-4)}{4p}\left(\tfrac{d_{s}(x,\mbb{K})}{2}\right)^{C_{3}}}\mbb{Q}(dx).
\end{align*}
Note that by Lemma \ref{lemma61}, we also have that $\phi_{s}\le C\Theta_{\mbb{K},s}$, for a constant $C>0$.

\begin{Lemma}
	Suppose the conditions in Corollary \ref{MTPoi} hold. Then, there exists a constant $C>0$ such that
	\begin{equation*}
	s^{2}\int_{\mbb{X}^{2}}\psi_{s}(x_{1},x_{2})^{2}\mbb{Q}^{2}(d(x_{1},x_{2}))\le C\Theta_{\mbb{K},s}.
	\end{equation*}
	\label{lemma66}
\end{Lemma}

\begin{proof}[Proof of Lemma~\ref{lemma66}]
	Without loss of generality, we assume $d_{s}(x_{1},\mbb{K})\ge d_{s}(x_{2},\mbb{K})$. Similar reasoning can be used for the other case. According to Lemma \ref{lemma61},
	\begin{equation*}
	\psi_{s}(x_{1},x_{2})\le Ce^{-C_{2}\frac{p-4}{4p}\max\{d_{s}(x_{1},x_{2}),d_{s}(x_{1},\mbb{K})\}^{C_{3}}}.
	\end{equation*}
	Let
	\begin{equation*}
	L_{x_{2},s}:=s\int_{\mbb{X}}\psi_{s}(x_{1},x_{2})^{2}\mbb{Q}(dx_{1}).
	\end{equation*}
	It suffices to show there exists a constant $C>0$ such that
	\begin{equation*}
	L_{x_{2},s}\le Ce^{-C_{2}\frac{p-4}{2p}\left(\frac{d_{s}(x_{2},\mbb{K})}{2}\right)^{C_{3}}}.
	\end{equation*}
		Let $r\coloneqq \frac{1}{2}d(x_{2},\mbb{K})$ and note the fact that $\max\{x,y\}\ge x,\max\{x,y\}\ge y$ for any $x,y$, then
		\begin{equation*}
		\begin{aligned}
		L_{x_{2},s}&\le Cs\int_{\mbb{X}}e^{-C_{2}\frac{p-4}{2p}\max\{d_{s}(x_{1},x_{2}),d_{s}(x_{1},\mbb{K})\}^{C_{3}}}\mbb{Q}(dx_{1})\\&\le Cs\int_{B_{x_{2}}(r)}e^{-C_{2}\frac{p-4}{2p}d_{s}(x_{1},\mbb{K})^{C_{3}}}\mbb{Q}(dx_{1})+Cs\int_{\mbb{X}\backslash B_{x_{2}}(r)}e^{-C_{2}\frac{p-4}{2p}d_{s}(x_{1},x_{2})^{C_{3}}}\mbb{Q}(dx_{1})\\&\coloneqq CL_{1}+CL_{2}.
		\end{aligned}
		\end{equation*}
		By the triangle inequality, $2r\le d(x_{2},\mbb{K})\le d(x_{1},x_{2})+d(x_{1},\mbb{K})$, then when $d(x_{2},x_{1})\le r$, $d(x_{1},\mbb{K})\ge r$. Therefore, according to Lemma \ref{lemma62} and lemma \ref{lemma64}, there exists a constant $C>0$ such that
		\begin{align*}
		L_{1}&\le s\int_{B_{x_{2}}(r)}e^{-C_{2}\frac{p-4}{2p}(s^{1/\omega}r)^{C_{3}}}\mbb{Q}(dx_{1}) \le s\kappa r^{\omega}e^{-C_{2}\frac{p-4}{4p}(s^{1/\omega}r)^{C_{3}}}\le Ce^{-C_{2}\frac{p-4}{4p}(s^{1/\omega}r)^{C_{3}}}.
		\end{align*}
		According to Lemma \ref{lemma63}, there exists a constant $C>0$ such that
		\begin{align*}
		L_{2}&\le s\cdot\frac{C}{s}e^{-C_{2}\frac{p-4}{4p}\left(s^{1/\omega}r\right)^{C_{3}}}=Ce^{-C_{2}\frac{p-4}{4p}\left(s^{1/\omega}r\right)^{C_{3}}}.
		\end{align*}
		Then, 
		\begin{align*}
		L_{x_{2},s}&\le Ce^{-C_{2}\frac{p-4}{4p}\left(s^{1/\omega}r\right)^{C_{3}}}=Ce^{-C_{2}\frac{p-4}{4p}\left(\frac{d_{s}(x_{2},\mbb{K})}{2}\right)^{C_{3}}},
		\end{align*}
    giving us the desired result. 
\end{proof}

\begin{Lemma}\label{lemma67}
	Suppose the conditions in Corollary \ref{MTPoi} hold. Then, there exists a constant $C>0$ such that
	\begin{equation*}
	s^{3}\int_{\mbb{X}}\left(\int_{\mbb{X}}\psi_{s}(x_{1},x_{2})\mbb{Q}(dx_{2})\right)^{2}\mbb{Q}(dx_{1})\le C\Theta_{\mbb{K},s}.
	\end{equation*}
	\label{lemma67}
\end{Lemma}

\begin{proof}[Proof of Lemma~\ref{lemma67}]
	We can prove this lemma in a similar way as Lemma \ref{lemma66}. Let
	\begin{equation*}
	L_{x_{1},s}':=s\int_{\mbb{X}}\psi_{s}(x_{1},x_{2})\mbb{Q}(dx_{2}).
	\end{equation*}
	Similar to $L_{x_{2},s}$, one can show there exists a constant $C>0$ such that
	\begin{equation*}
	L_{x_{1},s}'\le Ce^{-C_{2}\frac{p-4}{8p}\left(\frac{d_{s}(x_{1},\mbb{K})}{2}\right)^{C_{3}}}.
	\end{equation*}
	Therefore,
	\begin{align*}
	s^{3}\int_{\mbb{X}}\left(\int_{\mbb{X}}\psi_{s}(x_{1},x_{2})\mbb{Q}(dx_{2})\right)^{2}\mbb{Q}(dx_{1})&=s\int_{\mbb{X}}\left(s\int_{\mbb{X}}\psi_{s}(x_{1},x_{2})\mbb{Q}(dx_{2})\right)^{2}\mbb{Q}(dx_{1})\\&\le Cs\int_{\mbb{X}}e^{-C_{2}\frac{p-4}{4p}\left(\frac{d_{s}(x_{1},\mbb{K})}{2}\right)^{C_{3}}}\mbb{Q}(dx_{1})\\&\le C\Theta_{\mbb{K},s}.
	\end{align*}
\end{proof}

With the above ingredients in place, we are finally in a position to prove Corollary~\ref{MTPoi}.
 
\begin{proof}[Proof of Corollary \ref{MTPoi}]
Recall the definition of $\theta_1$, $\theta_2$ and $\theta_3$ respectively in \eqref{theta1}, \eqref{theta2} and \eqref{theta3}. According to Lemma \ref{lemma61}, Lemma \ref{lemma66} and Lemma \ref{lemma67}, there exists a constant $C>0$ such that
    \begin{align*}
	\theta_{1}\le C\frac{(\Theta_{\mbb{K},s})^{\frac{1}{2}}}{\Var F_{s}},\quad \theta_{2}\le C\frac{(\Theta_{\mbb{K},s})^{\frac{1}{2}}}{\Var F_{s}},\quad
	\phi_{s}\le C\Theta_{\mbb{K},s}.
	\end{align*}
Therefore, according to \eqref{missingthm}, we complete the proof. 
\end{proof}

\subsection{Proof of Theorem \ref{MTbino}} 
The proof of the Poisson case applies mutatis mutandis to the binomial case. 
We now highlight the main changes. First, recall that we do not have a similar result like Theorem \ref{general} for Poisson case due to the fact that there is no nice counterpart of the second order Poincar\'{e} inequality (see~\cite{last2016normal}). Hence, Theorem \ref{MTbino} cannot follow from Theorem \ref{general}. Instead, we use \cite[Theorem 4.2]{lachieze2019normal} that provides an auxiliary result for the binomial case. While~\cite{lachieze2019normal} provided a general result for marked binomial point process, we state the following result for the unmarked binomial point process.

\begin{Theorem}[\cite{lachieze2019normal}]\label{Thm 7.1}
Let $n\ge 3$ and let $F$ be a functional of a binomial point process $\eta_{n}$ with $\mbb{E}F(\xi_{n})^{2}<\infty$. Assume that there are constants $c,\rho \in (0,\infty)$ such that 
\begin{align*}
\mbb{E}|D_{x}F(\xi_{n-1-|\mbb{A}|}\cup \mbb{A})|^{4+\rho}\le c,\quad \mbb{Q}-a.e., x\in\mbb{X}, \mbb{A}\subset\mbb{X}, |\mbb{A}|\le 2.
\end{align*}
Then, there is a constant $C:=C(c,\rho)\in (0,\infty)$ such that
\begin{align*}
d_{K}\left(\frac{F-\mbb{E}F}{\sqrt{\Var F}},N\right)\le C(S_{1}'+S_{2}'+S_{3}'),
\end{align*}
with 
\begin{align}
\Gamma_{n}&':=n\int_{\mbb{X}}\mbb{P}(D_{x}F(\xi_{n-1})\neq 0)^{\frac{\rho}{8+2\rho}}\mbb{Q}(dx),\nonumber\\
\psi_{n}'(x,x')&:=\underset{\mbb{A}\subset\mbb{X}:|\mbb{A}|\le 1}{\sup}\mbb{P}(D_{x,x'}F(\xi_{n-2-|\mbb{A}|})\neq 0)^{\frac{\rho}{8+2\rho}},\nonumber\\
S_{1}'&:=\frac{n}{\Var F}\sqrt{\int_{\mbb{X}^{2}}\psi_{n}'(x,x')\mbb{Q}^{2}(d(x,x'))},\nonumber\\
S_{2}'&:=\frac{n^{\frac{3}{2}}}{\Var F}\sqrt{\int_{\mbb{X}}\left(\int_{\mbb{X}}\psi_{n}'(x,x')\mbb{Q}(dx')\right)^{2}\mbb{Q}(dx)},\nonumber\\
S_{3}'&:=\frac{(\Gamma_{n}')^{\frac{1}{2}}}{\Var F}+\frac{\Gamma_{n}'}{(\Var F)^{\frac{3}{2}}}+\frac{\Gamma_{n}'+(\Gamma_{n}')^{\frac{3}{2}}}{(\Var F)^{2}}. \label{S_{3}'}
\end{align}
\end{Theorem}

\begin{Remark}
The exponent of $\Gamma_{n}'$ in the third component of the
sum on the right hand side of  \eqref{S_{3}'} is different than that of Poisson case, due to the fact that it is not derived by the counterpart of the second order Poincar\'{e} inequality. In fact, it is obtained by \cite[Theorem 4.3]{lachieze2019normal}.
\end{Remark}

Now, based on Theorem \ref{Thm 7.1}, we start to prove Theorem \ref{MTbino}.
\begin{proof}[Proof of Theorem \ref{MTbino}]
The proof of the binomial point process is similar to that of the Poisson case, based on Theorem \ref{Thm 7.1} for binomial case. We treat Theorem \ref{Thm 7.1} for binomial as a counterpart of \eqref{missingthm} for Poisson noting $D_{x,x'}F_{n}=D_{x}F_{n}^{x'}-D_{x}F_{n}$. Starting with this, one can follow the same procedure to get the required counterparts of Lemma \ref{lemma61} to Lemma \ref{lemma67} by changing $s$ as $n$. This provides the desired result.
\end{proof}

\section{Proofs for Section~\ref{sec:app}}\label{proof2}

\subsection{Total Edge Length of $k$-NN}
\begin{proof}[Proof of Theorem~\ref{totalknn}]
We begin with the Poisson point process case and consider the statistic $F^{k\text{-NN}}_s(\mcal{P}_{s}) \coloneqq  \sum_{x\in \mcal{P}_{s}}f_{s}(x,\mcal{P}_{s}),$ with $f_s$ as defined in~\eqref{eq:fsforknn}.\\

\emph{Step 1:} From Example \ref{eg:knn}, $F^{k\text{-NN}}_s(\mcal{P}_{s})$ is a strongly stabilizing functional with the radius of stabilization $R_{x}=4R$, with $R$ being defined as follows: for each $t>0$, construct six disjoint equilateral triangles $T_{j}(t)$, $1\le j\le 6$, such that the origin is a vertex of each triangle, such that each triangle has edge length $t$ and such that $T_{j}(t)\subset T_{j}(u)$ whenever $t<u$. Then, define $R$ to be the minimum $t$ such that each triangle $T_{j}(t)$ contains at least $k
+1$ points from $\mcal{P}_{s}$. Consequently, we have 
\begin{align*}
\mbb{P}(R> r)\le \mbb{P}\left(\mcal{P}_{s}\left(\cup_{j=1}^{6}T_{j}(r)\right)\le 6k)\right)\le \mbb{P}\left(\mcal{P}_{s}\left(B_{x}\left({r\sqrt{3}}/{2}\right)\right)\le 6k\right),
\end{align*}
where $\mcal{P}_{s}\left(B_{x}\left({r\sqrt{3}}/{2}\right)\right)$ follows Poisson distribution with parameter $s\mbb{Q}\left(B_{x}\left({r\sqrt{3}}/{2}\right)\right)$. According to the assumption \eqref{lowerQ} and a Chernoff bound for Poisson tail, \cite[Lemma 1.2]{penrose2003random}, we have that there exits a constant $c'>0$ such that 
\begin{align*}\label{knnRx}
\mbb{P}(R> r)&\le \mbb{P}\left(\mcal{P}_{s}\left(B_{x}\left(\frac{r\sqrt{3}}{2}\right)\right)\le 6k\right)\\&\le \mbb{P}\left(\text{Poi}\left(cs\left(\frac{r\sqrt{3}}{2}\right)^{\omega}\right)\le 6k\right)\\&\le 6ke^{-c' sr^{\omega}}.
\end{align*}
This implies the radius of stabilization $R_{x}=4R$ decays exponentially. Furthermore, we set $\mbb{K}=\mbb{X}$ for the $\mathbb{K}-$exponential bound.

\emph{Step 2:} As for the bounded moment condition, according to \cite[Lemma 5.5]{lachieze2019normal}, for some $p>4$, the bounded moment condition holds.

\emph{Step 3:} As for the variance condition, by the assumption \eqref{convars}, it is satisfied.

Therefore, according to Corollary \ref{coros}, we end the proof for the Poisson case. For the binomial case, it is similar to the Poisson case by considering a Chernoff bound for the binomial distribution \cite[Lemma 1.1]{penrose2003random} and \cite[Lemma 5.6]{lachieze2019normal}. 
\end{proof}

\subsection{Shannon Entropy Estimation}

\begin{proof}[Proof of Theorem \ref{weightedknn}]

We start by replacing the biased center $\mbb{E}F_{n}^{\text{SE}}(\xi_{n})$ by the true parameter $H(q)$. By triangle inequality, we have that
\begin{align*}
d_{K}\left(\frac{F_{n}^{\text{SE}}(\xi_{n})-H(q)}{\sqrt{\Var F_{n}^{\text{SE}}(\xi_{n})}},N\right)&\le \underbrace{d_{K}\left(\frac{F_{n}^{\text{SE}}(\xi_{n})-\mbb{E}F_{n}^{\text{SE}}(\xi_{n})}{\sqrt{\Var F_{n}^{\text{SE}}(\xi_{n})}},N\right)}_{\coloneqq~d_1}\\&\quad+\underbrace{d_{K} \left(\frac{F_{n}^{\text{SE}}(\xi_{n})-H(q)}{\sqrt{\Var F_{n}^{\text{SE}}(\xi_{n})}},\frac{F_{n}^{\text{SE}}(\xi_{n})-\mbb{E}F_{n}^{\text{SE}}(\xi_{n})}{\sqrt{\Var F_{n}^{\text{SE}}(\xi_{n})}}\right)}_{\coloneqq~d_{2}}.
\end{align*}
We first apply Corollary \ref{corobino} to bound $d_{1}$, using the three step approach. 

\emph{Step 1:} Similar to the total edge length of $k$-NN, following the proof of \cite[Lemma 6.1]{penrose2001central}, $F_{n}^{\text{SE}}(\xi_{n})$
is strongly stabilizing with the radius of stabilization $R_{x}$ decaying exponentially and we take $\mbb{K}=\mbb{X}$.

\emph{Step 2:} As for the moment condition, again, by \cite[Lemma 5.6]{lachieze2019normal}, the bounded moment condition holds for $p>4$. 

\emph{Step 3:} For the variance, note that for $d_{1}$, we have
\begin{align*}
\frac{F_{n}^{\text{SE}}(\xi_{n})-\mbb{E}F_{n}^{\text{SE}}(\xi_{n})}{\sqrt{\Var F_{n}^{\text{SE}}(\xi_{n})}}\overset{d}{=}\frac{nF_{n}^{\text{SE}}(\xi_{n})-\mbb{E}\left(nF_{n}^{\text{SE}}(\xi_{n})\right)}{\sqrt{\Var \left(nF_{n}^{\text{SE}}(\xi_{n})\right)}},
\end{align*}
where $\overset{d}{=}$ means equal in distribution. Therefore, we can consider the variance condition for the $nF_{n}^{\text{SE}}(\xi_{n})$ instead. According to \cite[Lemma 7]{berrett2019efficient}, there exits a constant $C>0$ such that for $k_{0}^{*}\le k\le  k_{1}^{*}$,
\begin{equation}
    \sup_{n>0}\frac{n}{\mathrm{Var}\left(nF_{n}^{\text{SE}}(\xi_{n})\right)}\le C.\label{varw}
\end{equation}

By Corollary \ref{corobino}, we immediately have that $d_1 \leq C_{0}'(k,p)/\sqrt{n}$.  Since $k$ also diverges as $n$ goes to $\infty$, it is necessary to calculate the constant $C_{0}'(k,p)$ more explicitly.  Therefore, we derive the following lemma. Note that the following lemma is a refined version of Theorem \ref{Thm 7.1} as it reveals how the constant $C_{0}'(k,p)$ is related to the constant $p$ and the functional $F$ thus $k$. 


\begin{Lemma}
	Assume there are constants $c>0,p_{0}>0$ such that 
	\begin{equation*}
	\mbb{E}|D_{x}F(\xi_{n-1-|\mbb{A}|}\cup \mbb{A})|^{4+p_{0}}\le c,\ |\mbb{A}|\le 2.
	\end{equation*}
	Then, there exists some constant $C$ not depending on $n$ nor $F$ such that
	\begin{equation*}
	d_{K}\left(\frac{F-\mbb{E}F}{\sqrt{\Var F}},N\right)\le C(S_{1}+S_{2}+S_{3}+S_{4}+S_{5}),
	\end{equation*}
	where 
	\begin{align*}
	S_{1}&:=c^{\frac{2}{4+p_{0}}}\frac{n}{\Var F}\sqrt{\int_{\mbb{X}^{2}}\psi_{n}(x,x')\mbb{Q}^{2}(d(x,x'))},\\
    S_{2}&:=c^{\frac{2}{4+p_{0}}}\frac{n^{\frac{3}{2}}}{\Var F}\sqrt{\int_{\mbb{X}}\left(\int_{\mbb{X}}\psi_{n}(x,x')\mbb{Q}(dx')\right)^{2}\mbb{Q}(dx)},\\
    S_{3}&:=c^{\frac{2}{4+p_{0}}}\frac{\sqrt{\Gamma_{n}}}{\Var F},\\
	S_{4}&:=\bigg(\sqrt{3}\max\left\{4\sqrt{2}c^{\frac{1}{4+p_{0}}}\frac{\Gamma_{n}^{\frac{1}{2}}}{(\Var F)^{\frac{1}{2}}},\sqrt{2}c^{\frac{1}{4+p_{0}}}\frac{\Gamma_{n}^{\frac{1}{4}}}{(\Var F)^{\frac{1}{2}}}+1\right\}\\& \quad \quad\quad \quad+c^{\frac{1}{4+p_{0}}}\frac{\Gamma_{n}^{\frac{1}{4}}}{n^{\frac{1}{4}}(\Var F)^{\frac{1}{2}}}\bigg)c^{\frac{3}{4+p_{0}}}\frac{\Gamma_{n}}{(\Var F)^{\frac{3}{2}}}+c^{\frac{4}{4+p_{0}}}\frac{\Gamma_{n}}{(\Var F)^{2}},\\
	S_{5}&:=c^{\frac{3}{4+p_{0}}}\frac{\Gamma_{n}}{(\Var F)^{\frac{3}{2}}}.
	\end{align*}
	Here,
	\begin{align*}
	\Gamma_{n}&:=n\int_{\mcal{X}}\mbb{P}(D_{x}F(\xi_{n-1})\neq 0)^{\frac{p}{8+2p_{0}}}\mbb{Q}(dx),\\
\psi_{n}(x,x')&:=\underset{\mbb{A}\subset\mbb{X}:|\mbb{A}|\le 1}{\sup}\mbb{P}(D_{x,x'}F(\xi_{n-2-|\mbb{A}|}\cup\mbb{A})\neq 0)^{\frac{4}{8+2p_{0}}}.
	\end{align*}
	\label{kboundlemma}
\end{Lemma}

The proof of Lemma \ref{kboundlemma} follows in a straightforward manner by \cite[Proof of Theorem 4.2]{lachieze2019normal}.  
Now, we proceed to calculate moment bounds to see how $k$ is related to the constant $C_{0}'$. Let
\begin{equation*}
\zeta_{i}:=\sum_{j=1}^{k}w_{j}\log\left(\frac{(n-1)V_{d}\rho_{j,i}^{d}}{e^{\Psi(j)}}\right).
\end{equation*}
Following \cite[Lemma 5.6]{lachieze2019normal}, by Jensen's inequality, for $p=4+p_{0}$, $p_{0}>0$,
\begin{align*}
    &\mbb{E}|D_{y}F_{n}^{\text{SE}}(\xi_{n-1-|\mbb{A}|}\cup\mbb{A})|^{4+p_{0}}\\&=\mbb{E}\left|\zeta_{k}(y,\xi_{n-1-|\mbb{A}|}\cup\{y\}\cup\mbb{A})+\underset{x\in \xi_{n-1-|\mbb{A}|}\cup\mbb{A}}{\sum}D_{y}\zeta_{k}(x,\xi_{n-1-|\mbb{A}|})\right|^{4+p_{0}}\\& \le 4^{3+p_{0}}\mbb{E}|\zeta_{k}(y,\xi_{n-1-|\mbb{A}|}\cup\{y\}\cup\mbb{A})|^{4+p_{0}}+4^{3+p_{0}}\underset{x\in\mbb{A}}{\sum}\mbb{E}|D_{y}\zeta_{k}(x,\xi_{n-1-|\mbb{A}|})|^{4+p_{0}}\\&+4^{3+p_{0}}\underset{x\in\xi_{n-1-|\mbb{A}|}}{\sum}\mbb{E}|D_{y}\zeta_{k}(x,\xi_{n-1-|\mbb{A}|})|^{4+p_{0}}.
\end{align*}
Then, following \cite[Proof of Lemma 5.6]{lachieze2019normal}, the  constant $c$ in Lemma \ref{kboundlemma} satisfies
\begin{equation*}
    c\lesssim 4^{3+p_{0}}c_{a}+4^{3+p_{0}}c_{b}+4^{5+\frac{3}{2}p_{0}}c_{b}^{\frac{4+p_{0}}{4+2p_{0}}}k^{\frac{4}{4+2p_{0}}},
\end{equation*}
where $A \lesssim B$ means there is a constant $C>0$ such that $A\le CB$. Note that according to the definition of weights \eqref{w}, there only exist finitely many terms in the sum of $\zeta_{i}$ and according to \cite[Section 3.2]{singh2016finite},
\begin{align*}
c_{a}:=\mbb{E}|\zeta_{k}(y,\xi_{n-1-|\mbb{A}|}\cup\{y\}\cup\mbb{A})|^{4+p_{0}}\lesssim \|w\|_{\infty}^{4+p_{0}}\mbb{E}\left|\log \left(\frac{(n-1)V_{d}\rho_{k,i}^{d}}{e^{\Psi(k)}}\right)\right|^{4+p_{0}}<\infty,
\end{align*}
and similarly,
\begin{align*}
c_{b}:&=\mbb{E}|D_{y}\zeta_{k}(x,\xi_{n-1-|\mbb{A}|})|^{4+p_{0}}<\infty.
\end{align*}
Then,
\begin{equation*}
c\lesssim k^{\frac{p_{0}}{4+2p_{0}}},
\end{equation*}
and by \cite[Theorem 4.3]{lachieze2019normal} and \cite[Theorem 5.1]{lachieze2017new}, the constant $C$ in Lemma \ref{kboundlemma} does not depend on either $n$ or $k$.

Therefore, according to Corollary \ref{corobino} and Lemma \ref{kboundlemma},
\begin{align}
    d_{1}\lesssim \frac{1}{\sqrt{n}}\times c^{\frac{4}{4+p_{0}}}  \lesssim \frac{1}{\sqrt{n}} k^{\frac{4p_{0}}{(4+p_{0})(4+2p_{0})}}\le \sqrt{\frac{k}{n}} .
    \label{d1}
\end{align}

Next, we focus on $d_{2}$, which is related to the bias, $\mbb{E}F_{n}^{\text{SE}}(\xi_{n})-H(q)$. Let 
\begin{align*}
h^{w}:=\frac{F_{n}^{\text{SE}}(\xi_{n})-\mbb{E}F_{n}^{\text{SE}}(\xi_{n})}{\sqrt{\Var F_{n}^{\text{SE}}(\xi_{n})}}\quad\text{and}\quad\Delta h:=\frac{H(q)-\mbb{E}F_{n}^{\text{SE}}(\xi_{n})}{\sqrt{\Var F_{n}^{\text{SE}}(\xi_{n})}}.
\end{align*}
Then we have
\begin{align*}
d_{2}:&=d_{K} \left(\frac{F_{n}^{\text{SE}}(\xi_{n})-H(q)}{\sqrt{\Var F_{n}^{\text{SE}}(\xi_{n})}},\frac{F_{n}^{\text{SE}}(\xi_{n})-\mbb{E}F_{n}^{\text{SE}}(\xi_{n})}{\sqrt{\Var F_{n}^{\text{SE}}(\xi_{n})}}\right)\\&=\underset{t\in\mbb{R}}{\sup}~\left|\mbb{P}\left(\frac{F_{n}^{\text{SE}}(\xi_{n})-H(q)}{\sqrt{\Var F_{n}^{\text{SE}}(\xi_{n})}}\le t\right)-\mbb{P}\left(\frac{F_{n}^{\text{SE}}(\xi_{n})-\mbb{E}F_{n}^{\text{SE}}(\xi_{n})}{\sqrt{\Var F_{n}^{\text{SE}}(\xi_{n})}}\le t\right)\right|\\&=\underset{t\in\mbb{R}}{\sup}~\left|\mbb{P}(h^{w}\le t+\Delta h)-\mbb{P}(h^{w}\le t)\right|\\&\le 2d_{1}+\underset{t\in\mbb{R}}{\sup}~\left|\Phi(t+\Delta h)-\Phi(t)\right|\\&= 2d_{1}+2\left(\Phi\left(\left|\frac{\Delta h}{2}\right|\right)-\frac{1}{2}\right)\\&\le 2d_{1}+1-e^{-\frac{1}{2}(\Delta h)^{2}-\sqrt{\frac{2}{\pi}}|\Delta h|}\\& \lesssim d_{1}+|\Delta h|,
\end{align*}
where $\Phi$ is the c.d.f. of the standard normal distribution and actually, we applied a tight bound for standard normal distribution function according to \cite{mastin2013log}. Therefore,
\begin{equation*}
    d_{K} \left(\frac{F_{n}^{\text{SE}}(\xi_{n})-H(q)}{\sqrt{\Var F_{n}^{\text{SE}}(\xi_{n})}},N\right)\lesssim \max\{d_{1},|\Delta h|\}.
\end{equation*}

The bias of $F_{n}^{\text{SE}}(\xi_{n})$ satisfies the following bound according to \cite[Corollary 4]{berrett2019efficient}: for every $\epsilon>0$, 
\begin{equation}
    \underset{f\in\mcal{F}_{d,\theta}}{\sup}~|\mbb{E}_{f}\hat{H}_{n}^{w}-H(f)|=O\left(\max\left\{\left(\frac{k}{n}\right)^{\frac{\alpha}{\alpha+d}-\epsilon},\left(\frac{k}{n}\right)^{\frac{2(\lfloor\frac{d}{4}\rfloor+1)}{d}},\left(\frac{k}{n}\right)^{\frac{\beta}{d}}\right\}\right).\label{biasw}
\end{equation}
Note that $\alpha>d$ and $\beta>\frac{d}{2}$. With \eqref{varw} and \eqref{biasw} , by elementary algebraic manipulation we have that
\begin{equation}
\begin{aligned}
    |\Delta h|&\lesssim n^{-\frac{1}{2}}\times \left(\frac{k}{n}\right)^{\frac{1}{2}}\lesssim \frac{k^{\frac{1}{2}}}{n}.\label{Delta}
\end{aligned}
\end{equation}
Comparing \eqref{d1} and \eqref{Delta}, we obtain the desired result.

\end{proof}



\subsection{Euler Characteristic}
\begin{proof}[Proof of Theorem~\ref{thm:eulerc}]
	We follow the 3-step procedure. 
 
    \emph{Step 1:} According to Example \ref{egec}, the Euler characteristic is strongly stabilizing with the radius of stabilization $R_{x}=2r$. Clearly, as a constant, $R_{x}$ can be bounded exponentially as one can choose $c_{1}$ large enough and $c_{2},c_{3}>0$ such that for $0\le t\le 2r$,
	\begin{equation}
	1\le c_{1}e^{-c_{2}(n^{1/d}t)^{c_{3}}}.\label{exdEC}
	\end{equation}
    The similar argument also holds for Poisson case. Also, we take $\mbb{K}=[0,1]^{d}$.
	
	\emph{Step 2:} Moreover, for the bounded moment condition, note that for $p>4$,	there exists a constant $C(d,T)>0$ such that
	\begin{align}
	\underset{n> 0,x\in[0,1]^{d}}{\sup}~\mathbb{E}|D_{x}F_{n}^{\text{EC}}(\xi_{n})|^p\le & \mathbb{E}\bigg|\sum_{\ell =0}^{n}\#\{\sigma\in K_{r}\left(n^{\frac{1}{d}}(\xi_{n}\cup\{x\})\right):	\label{60} \\&\nonumber \quad \quad \quad \quad\text{$\sigma$ is an $\ell$-simplex intersecting with $\{x \}$}\}\bigg|^{p}\\ \nonumber
	\le & \left|\sum_{\ell =0}^{n}\sum_{\{j_{1},...,j_{\ell}\}\subset\{1,2,...,n\}}\left(\mathbb{P}\left(x_{j_{1}}\in B_{x}\left(n^{-\frac{1}{d}}r\right)\right)\right)^{\ell}\right|^{p}\\ \nonumber
	\le & \left|\sum_{\ell=0}^{n}\binom{n}{\ell}\frac{(C(d,T)\|q\|_{\infty}r^{d})^{\ell}}{n^{\ell}}\right|^{p}\\ \nonumber
	\le & \left|\sum_{\ell=0}^{n}\frac{(C(d,T)\|q\|_{\infty}r^{d})^{\ell}}{\ell!}\right|^{p}\\ \nonumber
	\le &(e^{C(d,T)\|q\|_{\infty}r^{d}})^{p}<\infty.
	\end{align}
	Also, by \cite[Lemma 4.1]{yogeshwaran2017random} and \cite[Proof of Lemma 4.2]{krebs2021approximation}, 
	\begin{equation*}
	\underset{n>0,x\in[0,1]^{d}}{\sup}~\mathbb{E}|F_{n}^{\text{EC}}(\mcal{P}_{n})|^p<\infty.
	\end{equation*}
	
	Moreover, consider
	\begin{align*}
	D_{x}F_{n}^{\text{EC}}(\xi_{n})^{x^*}&= \chi\left(K_{r}\left(n^{\frac{1}{d}}(\xi_{n}\cup\{x\}\cup\{x^{*}\})\right)\right)-\chi\left(K_{r}\left(n^{\frac{1}{d}}\xi_{n}\right)\right)\\&+\chi\left(K_{r}\left(n^{\frac{1}{d}}\xi_{n}\right)\right)-\chi\left(K_{r}\left(n^{\frac{1}{d}}(\xi_{n}\cup\{x^{*}\})\right)\right),
	\end{align*}
	and following a similar argument like \eqref{60}, we have
	\begin{equation*}
	\underset{n>0,x\in[0,1]^{d},x^{*}\in[0,1]^{d}}{\sup}~\mbb{E}|D_{x}F_{n}^{\text{EC}}(\xi_{n})^{x^*}|^{p}<\infty.
	\end{equation*}
	The bounded moment condition is satisfied and similar arguments can be utilized to show for Poisson case.
	
	\emph{Step 3:} According to \cite[Theorem 2.1]{penrose2001central} and \cite[proposition 4.6]{krebs2021approximation}, there exists a constant $C>0$ such that
	\begin{equation*}
	\underset{n>0}{\sup}~\frac{n}{\Var F_{n}^{\text{EC}}(\xi_{n})}\le C,
	\end{equation*}
    and it also holds for Poisson case.
    
    Therefore, according to Corollary \ref{coros} and Corollary \ref{corobino}, we complete the proof.
\end{proof}

\subsection{Edge Length Statistic of the Minimum Spanning Tree}
\begin{proof}[Proof of Theorem \ref{mst}]
According to \cite[Theorem 3.3]{penrose2005multivariate}, the total edge length $M(V)$ of the minimal spanning tree satisfies the strong stabilization with a radius of stabilization $R_{x}$ almost surely finite. Without knowing the tail probability of $R_{x}$, we need to use Theorem \ref{general}. Set
\begin{align*}
    A_{x_{1}}=B_{n}\cap \{x_{1}+n^{\alpha}B_{0}\},
\end{align*}
where $0<\alpha<1$ and $\{x+A\}:=\{x+y:y\in A\}$ for any set $A$. According to \cite[Proposition 3.7]{lachieze2020quantitative}, for any $q_{0}>0$ and $x\in B_{n}$, there exists a constant $C_{q_{0}}>0$ independent of $n,x$ such that uniformly
\begin{equation}
    \mbb{E}|D_{x}F_{B_{n}}^{\text{MST}}(\mcal{P}(\lambda))|^{q_{0}}\le C_{q_{0}}.
    \label{firstordermst}
\end{equation}
The same bound also holds for $D_{x}F_{B_{n}}^{y,\text{MST}}$

As for the second order cost, according to \cite[Proposition 3.11]{lachieze2020quantitative}\footnote{This proposition is provided for $q_{0}=1$. However, a closer examination of the proof reveals that it can be extended to any $q_{0}\ge 1$ in a straightforward manner.}, for any $q_{0}\ge 1$, there exist constants $E_{0}>0,E_{1}>0,E_{2}>0$ such that for any $x_{1}\in B_{n}$ with $d(x_{1},\partial B_{n})>n^{\alpha}$,
\begin{equation}
    \mathbb{E}|D_{x_{1}}F_{B_{n}}^{\text{MST}}(A_{x_1})-D_{x_{1}}F_{B_{n}}^{\text{MST}}(\mcal{P}(\lambda))|^{q_{0}}\le 
    \left\{
        \begin{aligned}
        &E_{0}n^{-E_{1}},\quad \text{if}\ d=2,\\
        &E_{0}(\log n)^{-E_{2}},\quad  \text{if}\ d\ge 3.
        \end{aligned}
        \right.
        \label{secondordermst}
\end{equation}
The same bound also holds for $D_{x_{1}}F_{B_{n}}^{y,\text{MST}}$. On the other hand, the variance bound is directly from \cite[Propostion 3.9]{lachieze2020quantitative}: there exists a constant $c>0$,
\begin{equation*}
    \Var F_{B_{n}}^{\text{MST}}(\mcal{P}(\lambda))\ge c|B_{n}|\asymp  n^{d},
\end{equation*}. 

Then, plugging these two bounds \eqref{firstordermst} and \eqref{secondordermst} in Theorem \ref{general}, we get that there exist constants $C_{\text{first}}>0$ and $C_{\text{second}}>0$ such that, 
\begin{align*}
    \sum_{j=1}^{2}b_{j}(x_{1},A_{x_{1}})^{\frac{1}{4}}\le C_{\text{first}},
\end{align*}
and for $x_{3}\in B_{n}$, $d(x_{3},\partial B_{n})>n^{\alpha}$ and $d(x_{1},x_{3})\ge n^{\alpha}$,
\begin{align*}
    \sum_{j=3}^{5}b_{j}(x_{3},x_{1},A_{x_{3}})^{\frac{1}{4}}\le \left\{
        \begin{aligned}
        &2E_{0}n^{-E_{1}},\quad \text{if}\ d=2,\\
        &2E_{0}(\log n)^{-E_{2}},\quad \text{if}\ d\ge 3,
        \end{aligned}
        \right.
\end{align*}
and for other $(x_{1},x_{3})\in B_{n}\times B_{n}$, the points near the boundary of $B_{n}$, similar to \cite[proof of Proposition 3.11]{lachieze2020quantitative}, we just apply uniform moment bounds for the flexible cost functions and get 
\begin{align*}
    \sum_{j=3}^{5}b_{j}(x_{3},x_{1},A_{x_{3}})^{\frac{1}{4}}\le C_{\text{second}}.
\end{align*}

Therefore, by the fact that 
\begin{align*}
    |\{(x,y)\in B_{n}^{2}:A_{x}\cap B_{n}\cap A_{y}\neq \emptyset\}|\asymp n^{d}n^{d\alpha},
\end{align*}
then we have, 
\begin{align*}
    \gamma_{1}'\lesssim\left\{
       \begin{aligned}
        &\frac{n^{\frac{d\alpha}{2}}}{n^{\frac{d}{2}}}C_{\text{second}}+2E_{0}n^{-E_{1}}, \quad \text{if}\ d=2,\\
        &\frac{n^{\frac{d\alpha}{2}}}{n^{\frac{d}{2}}}C_{\text{second}}+2E_{0}(\log n)^{-E_{2}}, \quad \text{if}\ d\ge 3.
        \end{aligned}
        \right.
\end{align*}
Similarly, one can derive
\begin{align*}
\gamma_{2}'\lesssim\left\{
       \begin{aligned}
        &\frac{n^{\frac{d\alpha}{2}}}{n^{\frac{d}{2}}}C_{\text{second}}+2E_{0}n^{-E_{1}}, \quad \text{if}\ d=2,\\
        &\frac{n^{\frac{d\alpha}{2}}}{n^{\frac{d}{2}}}C_{\text{second}}+2E_{0}(\log n)^{-E_{2}}, \quad \text{if}\ d\ge 3,
        \end{aligned}
        \right.
\end{align*}
and
\begin{align*}
\gamma_{3}',\gamma_{4}',\gamma_{5}'\lesssim \frac{1}{n^{\frac{d}{2}}},
\end{align*}
and
\begin{align*}
\gamma_{6}'\lesssim\left\{
       \begin{aligned}
        &\frac{n^{\frac{d\alpha}{2}}}{n^{\frac{d}{2}}}C_{\text{second}}+2E_{0}n^{-E_{1}}, \quad \text{if}\ d=2,\\
        &\frac{n^{\frac{d\alpha}{2}}}{n^{\frac{d}{2}}}C_{\text{second}}+2E_{0}(\log n)^{-E_{2}}, \quad \text{if}\ d\ge 3,
        \end{aligned}
        \right.
\end{align*}
Therefore, we complete the proof by invoking Theorem \ref{general}.
\end{proof}
\subsection*{Acknowledgement.} We sincerely thank Chinmoy Bhattacharjee and Matthias Schulte for several helpful discussions, clarifications regarding the literature, and valuable feedback on this manuscript. We gratefully acknowledge support for this project from the National Science Foundation via grant NSF-DMS-2053918. 

\bibliographystyle{alpha}
\bibliography{example}

\appendix

\section{Background on TDA}\label{sec:tda}

\begin{Definition}[Simplicial Complexes]\label{SC}
 An {\em abstract simplicial complex} over a (finite) vertex set $\mbb{A}_{n}:=\{a_1,\ldots,a_n\}$ is a collection $S$ of subsets of $\mathbb{A}_n$ with the properties that 
 \begin{itemize}
\item[(i)] $\{a_i\} \in S, i = 1,\ldots,n$,
\item[(ii)] $\sigma \in S$ and $\tau \subset \sigma$ implies that $\tau \in S$. 
\end{itemize}
Every $\tau \subset \sigma$ is called a {\em face} of $\sigma$. Every $\sigma \in S$ with $|\sigma| = \ell+1, \ell \ge 0,$ is called $\ell$-{\em simplex}. 
\end{Definition}

Note that the vertices do not necessarily have to be elements of a Euclidean space. If they are (affinely independent) elements of $\mathbb{R}^d$, one can think of every simplex of order $\ell \le d$ as a convex hull of $\ell+1$ (affinely independent) vertices, so that 0-simplices are points, 1-simplices are lines, 2-simplices are triangles, etc. The following two types of simplicial complexes, the Vietoris-Rips complex (VR complex) and the \v{C}ech complex, are widely used in TDA.

\begin{Definition}[VR Complex]\label{VR}
Following the definition \ref{SC}, let the vertex set $\mbb{A}_{n}$ be in a metric space with the metric $d$. Then, the VR complex $\text{VR}_{r}(\mbb{A}_{n})$ for a given postive real number $r>0$ is a collection of simplices, where a simplex $\sigma\in \text{VR}_{r}(\mbb{A}_{n})$ if and only if for any pair of vertices $a_{i},a_{j}\in \sigma$, $d(a_{i},a_{j})<r$.
    
\end{Definition}

\begin{Definition}[\v{C}ech Complex]\label{Cech}
Following the definition \ref{SC}, let the vertex set $\mbb{A}_{n}$ be in a metric space with the metric $d$. Then, the \v{C}ech complex $C_{r}(\mbb{A}_{n})$ for a given postive real number $r>0$ is a collection of simplices, where a simplex $\sigma:=\{a_{i}\}_{i\in I}\in C_{r}(\mbb{A}_{n})$ for some $I\subset\{1,2,...,n\}$ if and only if for $\underset{i\in I}{\cap} B_{a_{i}}\left(\frac{r}{2}\right)\neq \emptyset$.
    
\end{Definition}

\begin{Definition}[Filtrations]
A filtration $\mcal{S}$ of a simplicial complex $S$ is a nested sequence of simplicial complexes $\emptyset = S_0 \subset S_1 \subset \cdots \subset S_{I} = S$, where $S_i = S_{i-1} \cup \sigma_i, i = 1,\ldots, I$ for some $\sigma_i \in S.$  A filtration is thus equivalent to an ordering of the simplices in the complex. Usually, a filtration is given in form of a filtration function $\psi:S \to \mathbb{R}$ that assigns a real value $\psi(\sigma)$ to each simplex $\sigma \in S$. The filtration itself is then defined via $S(r) = \{\sigma \in S: \psi(\sigma) \le r\}$. Note that while $r$ is a continuous parameter, there are only finitely many values of $r$ at which the complex is changing for a simplex over a finite set of vertices as considered here. Additionally, the parameter $r$ here is called as the filtration parameter or the filtration time.
\end{Definition}
\begin{Remark}
The real number $r$ in Definition \ref{VR} and \ref{Cech} is called the filtration parameter/time for the VR complex and the \v{C}ech complex respectively.
\end{Remark}



\end{document}